\documentclass[12pt,CJK]{article}
\usepackage{amsfonts,amsmath,amssymb,amsthm,mathrsfs}
\usepackage{latexsym}
\usepackage{graphicx}
\usepackage{indentfirst}
\usepackage{color}
\usepackage{fancyhdr}
\usepackage[colorlinks,linktocpage,linkcolor=blue]{hyperref}

\theoremstyle{plain}
\newtheorem{thm}{Theorem}[section]

\newtheorem{lemma}[thm]{Lemma}

\newtheorem{remark}[thm]{Remark}
\numberwithin{equation}{section}

\theoremstyle{remark}

\def\Xint#1{\mathchoice
  {\XXint\displaystyle\textstyle{#1}}%
  {\XXint\textstyle\scriptstyle{#1}}%
  {\XXint\scriptstyle\scriptscriptstyle{#1}}%
  {\XXint\scriptscriptstyle\scriptscriptstyle{#1}}%
  \!\int}
\def\XXint#1#2#3{{\setbox0=\hbox{$#1{#2#3}{\int}$}
  \vcenter{\hbox{$#2#3$}}\kern-.5\wd0}}

\def\dashint{\Xint-}
\begin{document}
\allowdisplaybreaks
\pagestyle{myheadings}\markboth{$~$ \hfill {\rm Q. Xu,} \hfill $~$} {$~$ \hfill {\rm  } \hfill$~$}
\author{Qiang Xu
\thanks{Email: xuqiang@math.pku.edu.cn.},~
Weiren Zhao
\thanks{Email: zjzwr@math.pku.edu.cn},~
Shulin Zhou
\thanks{Email: szhou@math.pku.edu.cn}
% Here add name
\\
%EndAName
School of Mathematical Sciences, Peking University, \\
Beijing, 100871, PR China. \\
}

%\author{Weiren Zhao
%\thanks{Email: xuqiang@math.pku.edu.cn.}
%\thanks{This work was supported by the National Natural Science Foundation of China (Grant No. 11471147).}

\title{\textbf{Commutator Estimates for the Dirichlet-to-Neumann Map of Stokes Systems in Lipschitz Domains} }
\maketitle
\begin{abstract}
In the paper, we establish commutator estimates for the Dirichlet-to-Neumann map of Stokes systems in
Lipschitz domains. The approach is based on Dahlberg's bilinear estimates, and the results may be
regarded as an extension of \cite{BD4,Shen} to Stokes systems.
\end{abstract}

\section{Introduction and main results}

Let $\Omega\subset\mathbb{R}^d$ be a Lipschitz domain with $d\geq 3$.
It is well known that for any $f\in H^{1/2}(\partial\Omega;\mathbb{R}^d)$ with
the compatibility condition $\int_{\partial\Omega} n\cdot f dS = 0$, the Dirichlet problem
for the Stokes system
\begin{equation}\label{pde:1.1}
\left\{\begin{aligned}
\Delta u - \nabla q &= 0&\quad & \text{in} &~~\Omega,\\
\text{div} (u) &= 0 &\quad & \text{in} &~~\Omega,\\
          u &= f &\quad & \text{on} &~\partial\Omega
\end{aligned}\right.
\end{equation}
has a unique velocity $u$ in $V_N=\big\{v\in H^1(\Omega;\mathbb{R}^d):\text{div}(v) = 0\big\}$, and a unique pressure $q$ up to constants in $L^2(\Omega)$.
To make the following definition well-defined,
we may assume $\int_\Omega q(x) dx = 0$. The
Dirichlet-to-Neumann map $\Lambda: H^{1/2}(\partial\Omega;\mathbb{R}^d)
\to H^{-1/2}(\partial\Omega;\mathbb{R}^d)$ is defined by
\begin{equation}\label{f:1.1}
\big(\Lambda(f)\big)^\alpha = \frac{\partial u^\alpha}{\partial n} - n_\alpha q
\end{equation}
in a weak sense,
where $n=(n_1,\cdots,n_d)$ is the outward unit normal to $\partial\Omega$.
The right-hand side of $\eqref{f:1.1}$ denotes the conormal derivative of $u$ on $\partial\Omega$
(see for example \cite{EBFCEKGCV,OAL}). Furthermore, from the results in \cite{EBFCEKGCV},
one may show that $\|\Lambda(f)\|_{L^2(\partial\Omega)}\leq C\|f\|_{H^1(\partial\Omega)}$.

In the paper, we will study the $L^2$-theory of the commutator estimates for the Dirichlet-to-Neumann map $\eqref{f:1.1}$,
and the main results will be shown in the following.
\begin{thm}\label{thm:1.1}
Let $\Omega$ be a bounded Lipschitz domain, and
$f\in L^2(\partial\Omega;\mathbb{R}^d)$ satisfy the compatibility condition $\int_{\partial\Omega} n\cdot f dS = 0$. Suppose $(u,q)$ is the solution of $\eqref{pde:1.1}$ with boundary data $f$.
Then for any $\eta\in C^{0,1}(\partial\Omega)$ satisfying $\int_{\partial\Omega}n\cdot \eta f dS =0$,
we have
\begin{equation}\label{pri:1.1}
\big\|\Lambda(\eta f) - \eta\Lambda(f)\big\|_{L^2(\partial\Omega)}
\leq C\|\eta\|_{C^{0,1}(\partial\Omega)}\|f\|_{L^2(\partial\Omega)},
\end{equation}
where $C$ depends on $d$ and $\Omega$.
Particularly, in the case of $d=3$, the estimate
\begin{equation}\label{pri:1.2}
\big\|\Lambda(\eta f) - \eta\Lambda(f)\big\|_{L^2(\partial\Omega)}
\leq C\|\eta\|_{H^{1}(\partial\Omega)}\|f\|_{L^\infty(\partial\Omega)},
\end{equation}
also holds, where $C$ depends only on $\Omega$.
\end{thm}

\begin{remark}
\emph{In the proof of $\eqref{pri:1.2}$, the assumption that $d=3$ merely guarantees
the $L^\infty$-estimate $\|u\|_{L^\infty(\Omega)}\leq C\|f\|_{L^\infty(\partial\Omega)}$ is valid,
which is well known as the Agmon-Miranda maximum principle in the field of elliptic systems. Whether such the $L^\infty$-estimate
holds in Lipschitz domains for $d\geq 4$ remains an interesting open problem.}
\end{remark}

The estimates $\eqref{pri:1.1}$ and $\eqref{pri:1.2}$ are referred to as the commutator estimates.
The key step in the proof of Theorem $\ref{thm:1.1}$ is to establish the following Dahlberg's bilinear estimate
\begin{equation}\label{pri:1.3}
\begin{aligned}
\bigg|\int_{\Omega} \nabla u \cdot v dx\bigg|
&\leq C
\Bigg\{ \Big(\int_{\Omega} |\nabla u|^2 \delta(x) dx\Big)^{\frac{1}{2}}
+ \Big(\int_{\Omega} |q|^2 \delta(x) dx\Big)^{\frac{1}{2}}
\Bigg\} \\
&\qquad\qquad\qquad\times
\Bigg\{\Big(\int_{\Omega} |\nabla v|^2 \delta(x) dx\Big)^{\frac{1}{2}}
+
\Big(\int_{\partial\Omega} |(v)^*|^2 dS\Big)^{\frac{1}{2}}
\Bigg\}
\end{aligned}
\end{equation}
where $\delta(x) = \text{dist}(x,\partial\Omega)$, and
$v=(v_i^\alpha)\in H^1(\Omega;\mathbb{R}^{d\times d})$.
The notation $(v)^*$ in $\eqref{pri:1.3}$ represents the nontangential maximal function of
$v$ on $\partial\Omega$, defined by
\begin{equation*}
(v)^*(x) = \sup_{y\in\Gamma_{N_0}(x)}|v(y)|,
\qquad \Gamma_{N_0}(x) = \big\{y\in\Omega:|y-x|\leq N_0\text{dist}(y,\partial\Omega)\big\},
\end{equation*}
where $x\in\partial\Omega$, and $N_0$ is sufficiently  large.
The bilinear estimate was originally proved in \cite{BD3} for harmonic
functions in Lipschitz domains. In term of the elliptic system with variable coefficients,
it was established by S. Hofmann \cite{SH}, and by Z. Shen \cite{Shen}, respectively, for
different considerations. In fact, this work is much influenced by \cite{Shen}.

Compared to the bilinear estimate established for elliptic equations (see \cite{BD3,BD4,Shen,SH}), the estimate $\eqref{pri:1.3}$ has one more square function caused by the pressure term $q$, and how to
handle that term will be the main difficulty in the technical standpoint. In term of layer potential,
we have the key observation that $\Delta q = 0$ in $\mathbb{R}^d\setminus\partial\Omega$, which
leads two important facts.
One is that the square function of $q$ may be controlled by the boundary data (see Lemma $\ref{lemma:2.1}$), which is based on the equivalence between the square function and
the nontangential maximal function (see \cite{RS,BDCKJPGV}).
The other is that $|q(x)|^2\delta(x)dx$ could be a Carleson measure provided
the velocity term $u$ is bounded.
Although these results may probably be known by experts,
a rigorous proof seems to have considerable merit, and benefits the readers.

The commutator is also denoted by $\big[\Lambda,\eta\big]f = \Lambda(\eta f) - \eta\Lambda(f)$ on
$\partial\Omega$,
and in view of $\eqref{f:1.1}$, it is not hard to derive the following identity,
which the proof of Theorem $\ref{thm:1.1}$ begin with,
\begin{equation}\label{f:1.2}
\begin{aligned}
\int_{\partial\Omega}\big[\Lambda,\eta\big]f\cdot h dS
& = \int_{\partial\Omega} [\Lambda(\eta f)]^\alpha h^\alpha dS - \int_{\partial\Omega} \eta
[\Lambda(f)]^{\alpha}h^\alpha dS \\
& = \int_{\Omega} u^\alpha \nabla \eta \cdot \nabla h^\alpha dx
-  \int_{\Omega} \nabla u^\alpha \cdot \nabla \eta h^\alpha dx
+ \int_{\Omega} q\nabla_\alpha \eta  h^\alpha dx
- \int_{\Omega} \pi\nabla_\alpha \eta u^\alpha dx,
\end{aligned}
\end{equation}
where $(h,\pi)\in H^1(\Omega;\mathbb{R}^d)\times L^2(\Omega)$ satisfies $\Delta h = \nabla\pi$ and $\text{div}(h) = 0$ in $\Omega$ and $(h)^*\in L^2(\partial\Omega)$. We mention that the summation
convention for repeated indices is used throughout.
Besides, the extension of $\eta$
is still denoted by itself, since proving $\eqref{pri:1.1}$ and $\eqref{pri:1.2}$ requires
the different way in extension of $\eta$. The former needs
$\|\nabla \eta\|_{L^\infty(\Omega)}\leq C\|\eta\|_{C^{0,1}(\partial\Omega)}$,
and $|\nabla^2\eta|^2\delta(x)dx$ is a Carleson measure.
The latter asks for an harmonic extension of $\eta$ to $\Omega$.

Observing the identity, the first two terms in the second line of $\eqref{f:1.2}$ will be
reduced to prove the Dahlberg's bilinear estimate $\eqref{pri:1.3}$, while to bound the following
integrals
\begin{equation*}
\bigg|\int_{\Omega} q\nabla_\alpha \eta  h^\alpha dx\bigg|
\quad
\text{and}
\quad
\bigg|\int_{\Omega} \pi\nabla_\alpha \eta u^\alpha dx\bigg|
\end{equation*}
is much involved, in which we borrowed some ideas from \cite{RS}.
We end the paragraph by mention that the core aim of the computation is to
control the conormal derivative through the tangential derivative of the solution on account of the equation satisfied by the solution itself.
In fact, in the computation we find that transferring the derivative seems to move the Carleson measure from one place to another.

In order to quick understand such the communicator estimates, for example $\eqref{pri:1.1}$,
we employ Fourier transformation to establish it in the upper half-space $\mathbb{R}_+^2$.
Since we will not pursue this approach, the concrete statements will be shown in the appendix of the paper. We mention that the estimates $\eqref{pri:1.1}$ and $\eqref{pri:1.2}$
could be established through the layer potential methods, which had been shown by S. Hofmann \cite{SH} in detail, and by Z. Shen concisely in \cite{Shen}.

To the best knowledge of the authors, the Dirichlet-to-Neumann map plays an fundamental role in the classical  Calder\'on problem, whose study may go back to the celebrated work \cite{CAP1}. It has many practical applications, notably to geophysics and medical imaging. We hope our results may be further
applied to the study of fluid mechanics. For more knowledges on this subject,
we refer the readers to \cite{JBAFME,CAP,BD1,BD2,BDCK,DJSK,MMMW,Shen1,Shen2} for more details and references therein.

We organize the paper as follows.
The proof of Theorem $\ref{thm:1.1}$ is given in Section 4,
while Section 3 is devoted to discuss the special case $\Omega =\mathbb{R}^d_+$, which is prepared for
readers who are unacquainted with Stokes systems or harmonic analysis,
and experts can skim this part. Some important lemmas related to the square function, nontangential
maximal functions and Carleson measures are presented in Section 2.

\section{Preliminaries}
The following lemma is related to extensions of Lipschitz functions.
\begin{lemma}\label{lemma:2.1}
Let $\Omega$ be a bounded Lipschitz domain. Let $\eta\in C^{0,1}(\partial\Omega)$ be a Lipschitz function on $\partial\Omega$. Then there exists $G\in C^{0,1}(\overline{\Omega})\cap C^\infty(\Omega)$ such
that $G=\eta$ on $\partial\Omega$, $\|\nabla G\|_{L^\infty(\Omega)}\leq C\|\eta\|_{C^{0,1}(\partial\Omega)}$,
and $d\nu = |\nabla^2 G|\delta(x) dx$ is a Carleson measure on $\Omega$ with norm
$\|\nu\|_{\mathcal{C}}\leq C\|\eta\|_{C^{0,1}(\partial\Omega)}$, where $C$ depends only on $\Omega$.
\end{lemma}

\begin{proof}
The proof may be found in \cite[Lemma 4.1]{Shen}.
\end{proof}

\begin{remark}\label{re:2.2}
\emph{In the following context, we do not distinguish the notation $G$ from $\eta$,
and always use $\eta$ instead of $G$.}
\end{remark}

\begin{thm}\label{thm:2.1}
Let $\Omega$ be a bounded Lipschitz domain in $\mathbb{R}^d$ with $d\geq 3$.
Suppose $(u,q)$ is the solution of $\Delta u = \nabla q$ and $\emph{div}(u) = 0$ in $\Omega$, and
$(u)^*\in L^2(\partial\Omega)$. Then we have
\begin{equation}\label{pri:2.4}
\int_\Omega |\nabla u(x)|^2\delta(x) dx
\leq C\int_{\partial\Omega}|(u)^*|^2 dS
\end{equation}
and there exists a function $\tilde{q}$ such that $\tilde{q}-q\in \mathbb{R}$, and
\begin{equation}\label{pri:2.5}
\int_\Omega |\nabla \tilde{q}(x)|^2[\delta(x)]^3 dx
\leq  \int_\Omega |\tilde{q}(x)|^2\delta(x) dx
\leq C\int_{\partial\Omega}|u|^2 dS
\end{equation}
where $C$ depends only on $d$ and $\Omega$.
\end{thm}

\begin{proof}
The proofs may be found in \cite[Theorem A.1]{RS} and
\cite[Lemma A.9]{RS}, as well as $\cite{BDCKJPGV}$.
\end{proof}

\begin{remark}\label{re:2.1}
\emph{If $\Omega = \mathbb{R}^d_+$, the estimates $\eqref{pri:2.4}$ and $\eqref{pri:2.5}$ will still be true, provided that the solution $(u,q)$ satisfies an additional vanishing condition as $|x|$ goes to infinity.
In fact, the proof will be simpler than that given in \cite{RS}.
We mention that it is clear to see that $\tilde{q}$ could be
replaced by $q$, provided we introduce the additional condition $\int_\Omega q dx= 0$. }
\end{remark}

\begin{lemma}
[Key identity]
Assume $\eta,f$ are given as in Theorem $\ref{thm:1.1}$. Let
$(u,q)\in H^1(\Omega;\mathbb{R}^d)\times L^2(\Omega)$ be the solution of $\eqref{pde:1.1}$.
Then we have the following identity
\begin{equation}\label{pri:2.1}
\begin{aligned}
\int_{\partial\Omega}\big[\Lambda,\eta\big]f\cdot h dS
& = \int_{\partial\Omega} \Lambda(\eta f^\alpha)h^\alpha dS - \int_{\partial\Omega} \eta\Lambda(f^\alpha)h^\alpha dS \\
& = \int_{\Omega} u^\alpha \nabla \eta \cdot \nabla h^\alpha dx
-  \int_{\Omega} \nabla u^\alpha \cdot \nabla \eta h^\alpha dx
+ \int_{\Omega} q\nabla_\alpha \eta  h^\alpha dx
- \int_{\Omega} \pi\nabla_\alpha \eta u^\alpha dx,
\end{aligned}
\end{equation}
where $(h,\pi)\in H^1(\Omega;\mathbb{R}^d)\times L^2(\Omega)$ satisfies $\Delta h = \nabla\pi$ and $\emph{div}(h) = 0$ in $\Omega$ and $(h)^*\in L^2(\partial\Omega)$.
\end{lemma}

\begin{proof}
The main tool is the Green formula, and we provide a proof for the sake of the completeness.
It is fine to let $(w,p)$ satisfy $[\Lambda(\eta f)]^\alpha = (\partial w^\alpha/\partial n) -
n_\alpha p$, and be the solution of
\begin{equation}\label{pde:2.1}
 \Delta w = \nabla p, \qquad \text{div}(w) = 0 \quad \text{in}~\Omega,
 \qquad\text{and}\quad w = \eta f \quad\text{on}~\partial\Omega.
\end{equation}
We have the following computation
\begin{equation}\label{f:2.1}
\begin{aligned}
\int_{\partial\Omega}\frac{\partial w^\alpha}{\partial n}h^\alpha dS
& = \int_{\partial\Omega}\frac{\partial h^\alpha}{\partial n} w^\alpha dS
+ \int_\Omega \Big(\Delta w^\alpha h^\alpha - w^\alpha \Delta h^\alpha \Big) dx \\
&= \int_\Omega \text{div}\big(\nabla h^\alpha u^\alpha \eta\big) dx
+ \int_\Omega \Big(\nabla_\alpha p h^\alpha - w^\alpha \Delta h^\alpha \Big) dx \\
& = \int_\Omega \big(u^\alpha \eta - w^\alpha\big)\Delta h^\alpha dx
+ \int_\Omega \big(u^\alpha \nabla \eta + \eta \nabla u^\alpha\big)\cdot \nabla h^\alpha dx
+ \int_\Omega \nabla_\alpha p h^\alpha dx,
\end{aligned}
\end{equation}
where the second equality follows from the divergence theorem. Using the equation $\eqref{pde:2.1}$
and $\Delta h = \nabla\pi$ and $\text{div}(h) = 0$ in $\Omega$,
\begin{equation}\label{f:2.2}
\begin{aligned}
\int_\Omega \big(u^\alpha \eta - w^\alpha\big)\Delta h^\alpha dx
&=\int_\Omega \nabla_\alpha \pi (u^\alpha \eta - w^\alpha) dx
= -\int_\Omega \pi \nabla_\alpha \eta u^\alpha dx, \\
\int_{\partial\Omega} n_\alpha p h^\alpha dS &= \int_{\Omega}\text{div}(ph) dx
= \int_{\Omega}\nabla_\alpha ph^\alpha dx,
\end{aligned}
\end{equation}
where we recall that $u=f$ on $\partial\Omega$.
Combining the identities $\eqref{f:2.1}$ and $\eqref{f:2.2}$, we have
\begin{equation}\label{f:2.3}
\begin{aligned}
\int_{\partial\Omega}\Lambda(\eta f)h dS
&= \int_{\partial\Omega} \Big(\frac{\partial w^\alpha}{\partial n}
- n_\alpha p\Big)h^\alpha dS \\
& = \int_\Omega \big(u^\alpha \nabla \eta + \eta \nabla u^\alpha\big)\cdot \nabla h^\alpha dx
-\int_\Omega \pi \nabla_\alpha \eta u^\alpha dx
\end{aligned}
\end{equation}

Then by the same token, we have the following expression
\begin{equation*}
\int_{\partial\Omega} \eta\Lambda(f)h dS =
\int_\Omega \nabla \eta \cdot \nabla u^\alpha h^\alpha dx
+ \int_\Omega \eta \nabla u^\alpha\cdot \nabla h^\alpha dx
- \int_\Omega q \nabla_\alpha \eta h^\alpha dx,
\end{equation*}
which together with the identity $\eqref{f:2.3}$ gives the desired result $\eqref{pri:2.1}$, and we
have completed the proof.
\end{proof}

\begin{lemma}\label{lemma:2.2}
Let $\Omega\subset\mathbb{R}^d$ be a bounded Lipschitz domain. Suppose that $(u,q)$ satisfies
$\Delta u = \nabla q$ and $\emph{div}(u) = 0$ in $\Omega$, and $|u|\in L^\infty(\Omega)$.
Then $|\nabla u(x)|^2\delta(x)dx$ and $|q(x)|^2\delta(x)dx$ will be the Carleson measures.
Moreover, for any $(v)^*\in L^2(\partial\Omega)$ there holds
\begin{equation}\label{pri:2.6}
\max\bigg\{
\int_\Omega |v|^2 |\nabla u|^2\delta(x) dx,
~\int_\Omega |v|^2 |q|^2 \delta(x) dx
\bigg\}\leq
C\|u\|_{L^\infty(\Omega)}^2\|(v)^*\|_{L^2(\partial\Omega)}^2,
\end{equation}
where $\delta(x)=\emph{dist}(x,\partial\Omega)$, and $C$ depends only on $d$ and $\Omega$.
\end{lemma}

\begin{proof}
For a cube $Q$ in $\mathbb{R}^{d-1}$, we define the tent over $Q$ to be the cube
$T(Q) = Q \times (0,l(Q)]$, also denoted by $Q^*$. Then we take the Whitney decomposition of $\Omega$,
and let $Q^*$ be one of such cubes, which satisfy
the property that $3Q^*\subset\Omega$ and $l(Q)\approx\text{dist}(Q^*,\partial\Omega)$.
Hence, in order to verify $d\nu_u=|\nabla u|^2\delta(x)dx$ is
a Carleson measure, it suffices to prove
\begin{equation}\label{f:2.16}
d\nu_u(T(Q)) = \int_{T(Q)} |\nabla u(x)|^2\delta(x)dx  \leq C\|u\|_{L^\infty(\Omega)}^2[l(Q)]^{d-1}.
\end{equation}
Since $\Delta u = \nabla q$ in $\Omega$, we have the interior estimates
\begin{equation*}
|\nabla u(x)| \leq \frac{C}{\delta(x)}\Big(\dashint_{B(x,\delta(x))}|u(y)|^2dy\Big)^{1/2}
\end{equation*}
and this implies that
\begin{equation*}
\int_{T(Q)} |\nabla u|^2 \delta(x) dx \leq C[l(Q)]^{-1}
\int_{T(Q)}\dashint_{B(x,l(Q))}|u(y)|^2dy dx
\leq C\|u\|_{L^\infty(\Omega)}^2[l(Q)]^{d-1}.
\end{equation*}

We now proceed to show $d\nu_q = |q|^2\delta(x)dx$ is another Carleson measure. The original idea will
be found in \cite[pp.1203-1204]{RS}, and we provide the proof for the sake of completeness.
Let $\omega_d$ denote the surface area of the unit sphere in $\mathbb{R}^d$. We introduce
the corresponding fundamental solution $(\Gamma_{ij},\Pi^i)$ of the Stokes system, which is given by
\begin{equation*}
\Gamma_{ij}(x) = \frac{1}{2\omega_d}
\bigg\{\frac{\kappa_{ij}}{(n-2)|x|^{d-1}}+\frac{x_ix_j}{|x|^d}\bigg\},
\qquad\quad
\Pi^i(x) = \frac{1}{\omega_d}\frac{x_i}{|x|^d}
\end{equation*}
(see for example \cite{OAL}). Then, in view of \cite[Section 3]{EBFCEKGCV}, $u$ can be represented in terms of a double layer potential
\begin{equation*}
 u^i(x) = \int_{\partial\Omega}
 \Big\{\frac{\partial}{\partial y_k}\big\{\Gamma_{ij}(x-y)\big\}n_k(y) - \Pi^i(x-y)n_j(y)\Big\}\phi_j(y)dS(y)
\end{equation*}
where $\|\phi\|_{L^2(\partial\Omega)}\leq C\|u\|_{L^\infty(\Omega)}$. By a standard computation,
we have
\begin{equation*}
\Delta u^i(x) = -\frac{\partial}{\partial x_i}\frac{\partial}{\partial x_k}
\int_{\partial\Omega}\frac{x_j-y_j}{\omega_d|x-y|^d}n_k(y)\phi_j(y)dS(y)
= \frac{\partial}{\partial x_i}\big\{\tilde{q}\big\},
\end{equation*}
where
\begin{equation*}
\tilde{q}(x) = - \frac{\partial W_k}{\partial x_k},
\quad\text{and}\quad
W_k(x) = \int_{\partial\Omega}\frac{x_j-y_j}{\omega_d|x-y|^d}n_k(y)\phi_j(y)dS(y).
\end{equation*}
From $\nabla (q-\tilde{q}) = 0$, it follows that $q-\tilde{q}\in\mathbb{R}$, and it is not hard to
observe $\Delta W_k = 0$ in $\Omega$.
Hence,
\begin{equation*}
\int_\Omega |\tilde{q}|^2\delta(x)dx
\leq \sum_{k=1}^d\int_\Omega|\nabla W_k|^2\delta(x)dx
\leq C\sum_{k=1}^d\int_{\partial\Omega}|(W_k)^*|^2dS
\leq C\|\phi\|_{L^2(\partial\Omega)}^2\leq C\|u\|_{L^\infty(\Omega)}^2,
\end{equation*}
where the fourth inequality follows from \cite{V}, and
and this implies that
\begin{equation}\label{f:2.17}
\nu_q(T(Q)) = \int_\Omega |q|^2\delta(x)dx \leq C\|\phi\|_{L^2(\partial\Omega)}^2\leq C\|u\|_{L^\infty(\Omega)}^2|Q|.
\end{equation}

Consequently, combining the estimates $\eqref{f:2.16}$, $\eqref{f:2.17}$ and
\cite[Corollary 7.3.6]{JD} leads to the desired estimate $\eqref{pri:2.6}$, and we are done.
\end{proof}

\section{Special case: $\Omega = \mathbb{R}_+^d$}

Let $\Omega = \mathbb{R}^d_+ = \{(x^\prime,t)\in\mathbb{R}^d:t>0\}$ be the upper half-space in $\mathbb{R}^d$, and
$\partial\Omega = \{(x^\prime,0), x^\prime\in\mathbb{R}^{d-1}\}=\mathbb{R}^{d-1}$.
In the section, we extend the investigation of Section $\ref{section:5}$ to the higher dimensional space $\mathbb{R}_+^d$
with $d\geq 3$ but using a different methods. Since the main techniques applied to Lipschitz domains have already appeared in such the case, we take it as
an example to make the main idea clear in the full proof of Theorem $\ref{thm:1.1}$.

\begin{thm}\label{thm:3.1}
Let $f=(f^\alpha)\in L^2(\mathbb{R}^{d-1};\mathbb{R}^d)$ satisfy the compatibility condition $\int_{\partial\Omega} f^d dS = 0$. Suppose $(u,q)$ is the solution of $\eqref{pde:1.1}$ with boundary data $f$.
Then for any $\eta\in C^{0,1}_0(\mathbb{R}^{d-1})$ such that the quantity $\Lambda(\eta f)$ is well-defined,
we have
\begin{equation}\label{pri:3.3}
\big\|\Lambda(\eta f) - \eta\Lambda(f)\big\|_{L^2(\mathbb{R}^{d-1})}
\leq C\|\eta\|_{C^{0,1}(\mathbb{R}^{d-1})}\|f\|_{L^2(\mathbb{R}^{d-1})},
\end{equation}
where $C$ depends only on $d$.
\end{thm}

\begin{lemma}
Let $(u,q)$ be the solution of $\Delta u = \nabla q$
and $\emph{div}(u) = 0$ in $\mathbb{R}_+^d$
with $(u)^*\in L^2(\mathbb{R}^{d-1})$.
Assume that $\eta\in C^{0,1}(\mathbb{R}^{d-1})$ has compact support, and
$h(x)$ vanishes as $|x|$ goes to infinity.
Then we have
\begin{equation}\label{pri:2.2}
\begin{aligned}
\int_{\mathbb{R}^d_+} q \nabla_\alpha \eta h^\alpha dx^\prime dt
&= -\int_{\mathbb{R}^d_+} t \frac{\partial^2 \eta}{\partial x_\alpha \partial t}  h^\alpha q dx^\prime dt
- \int_{\mathbb{R}^d_+} t\frac{\partial \eta}{\partial x_\alpha}\frac{\partial h^\alpha}{\partial t} q dx^\prime dt \\
& + \frac{1}{2} \int_{\mathbb{R}^d_+} t^2\frac{\partial^2 \eta}{\partial x_\alpha\partial t}\frac{\partial q}{\partial t} h^\alpha dx^\prime dt
+ \frac{1}{2} \int_{\mathbb{R}^d_+} t^2\frac{\partial \eta}{\partial x_\alpha}\frac{\partial q}{\partial t} \frac{\partial h^\alpha}{\partial t} dx^\prime dt \\
& - \frac{1}{2} \int_{\mathbb{R}^d_+} t^2\frac{\partial^2 \eta}{\partial x_\alpha\partial x_i}\frac{\partial q}{\partial x_i} h^\alpha dx^\prime dt
- \frac{1}{2} \int_{\mathbb{R}^d_+} t^2\frac{\partial \eta}{\partial x_\alpha}\frac{\partial q}{\partial x_i} \frac{\partial h^\alpha}{\partial x_i} dx^\prime dt
\end{aligned}
\end{equation}
where $i=1,\cdots,d-1$. Moreover, there holds
\begin{equation}\label{pri:3.1}
\bigg|\int_{\mathbb{R}^d_+} q \nabla_\alpha \eta h^\alpha dx^\prime dt\bigg|
\leq C\|\eta\|_{C^{0,1}(\mathbb{R}^{d-1})}\|u\|_{L^2(\mathbb{R}^{d-1})}\|(h)^*\|_{L^2(\mathbb{R}^{d-1})},
\end{equation}
where $C$ depends only on $d$.
\end{lemma}

\begin{proof}
Taking integration by parts with respect to $t$ variable, we have
\begin{equation}\label{f:2.4}
\begin{aligned}
\int_{\mathbb{R}^d_+} q\nabla_\alpha \eta h^\alpha dx^\prime dt
& = -\int_{\mathbb{R}^d_+} t \frac{\partial}{\partial t}
\bigg\{\frac{\partial \eta}{\partial x_\alpha} h^\alpha   q\bigg\} dx^\prime dt \\
& = -\int_{\mathbb{R}^d_+} t\frac{\partial^2 \eta}{\partial x_\alpha\partial t} h^\alpha q dx^\prime dt
-\int_{\mathbb{R}^d_+} t\frac{\partial \eta}{\partial x_\alpha}\frac{\partial h^\alpha}{\partial t} q
dx^\prime dt
-\int_{\mathbb{R}^d_+} t\frac{\partial \eta}{\partial x_\alpha}\frac{\partial q}{\partial t} h^\alpha
dx^\prime dt.
\end{aligned}
\end{equation}
We now turn to calculate the last term in the second line of $\eqref{f:2.4}$, and by the same token,
\begin{equation}\label{f:2.5}
\begin{aligned}
\int_{\mathbb{R}^d_+} t\frac{\partial \eta}{\partial x_\alpha}\frac{\partial q}{\partial t} h^\alpha
dx^\prime dt
&= -\frac{1}{2}
\int_{\mathbb{R}^d_+} t^2
\bigg\{\frac{\partial^2 \eta}{\partial x_\alpha\partial t}\frac{\partial q}{\partial t}
h^\alpha
+ \frac{\partial \eta}{\partial x_\alpha}\frac{\partial q}{\partial t}
\frac{\partial h^\alpha}{\partial t}
+ \frac{\partial \eta}{\partial x_\alpha}\frac{\partial^2 q}{\partial t^2} h^\alpha\bigg\}dx^\prime dt.
\end{aligned}
\end{equation}
Noting that $\Delta q = 0$ in $\Omega$, we have $\frac{\partial^2 q}{\partial t^2}
= \sum_{i=1}^{d-1}\frac{\partial^2 q}{\partial x_i^2}$, and by substituting it into the third term
in the right-hand side of $\eqref{f:2.5}$ leads to
\begin{equation*}
\begin{aligned}
\int_{\mathbb{R}^d_+} t^2\frac{\partial \eta}{\partial x_\alpha}\frac{\partial^2 q}{\partial t^2} h^\alpha
dx^\prime dt
&= - \sum_{i=1}^{d-1}\int_{\mathbb{R}^d_+} t^2\frac{\partial \eta}{\partial x_\alpha}\frac{\partial^2 q}{\partial x_i^2} h^\alpha
dx^\prime dt \\
& = \sum_{i=1}^{d-1}\int_{\mathbb{R}^d_+} t^2
\bigg\{\frac{\partial^2 \eta}{\partial x_\alpha
\partial x_i}\frac{\partial q}{\partial x_i} h^\alpha
+ \frac{\partial \eta}{\partial x_\alpha}\frac{\partial q}{\partial x_i}\frac{\partial h^\alpha}{\partial x_i} \bigg\}dx^\prime dt.
\end{aligned}
\end{equation*}
Then inserting the above formula into $\eqref{f:2.5}$, we arrive at
\begin{equation}\label{f:2.6}
\begin{aligned}
\int_{\mathbb{R}^d_+} t\frac{\partial \eta}{\partial x_\alpha}\frac{\partial q}{\partial t} h^\alpha
dx^\prime dt
= -\frac{1}{2}
\int_{\mathbb{R}^d_+} t^2
&\bigg\{\frac{\partial^2 \eta}{\partial x_\alpha\partial t}\frac{\partial q}{\partial t}
h^\alpha
+ \frac{\partial \eta}{\partial x_\alpha}\frac{\partial q}{\partial t}
\frac{\partial h^\alpha}{\partial t} \\
&+ \frac{\partial^2 \eta}{\partial x_\alpha
\partial x_i}\frac{\partial q}{\partial x_i} h^\alpha
+ \frac{\partial \eta}{\partial x_\alpha}\frac{\partial q}{\partial x_i}\frac{\partial h^\alpha}{\partial x_i}\bigg\}dx^\prime dt.
\end{aligned}
\end{equation}
Up to now, the desired identity $\eqref{pri:2.2}$ follows from
$\eqref{f:2.4}$ and $\eqref{f:2.6}$, and then we turn to estimate $\eqref{pri:3.1}$.

By the identity $\eqref{pri:2.2}$, it is not hard to derive
\begin{equation}\label{f:3.1}
\begin{aligned}
\bigg|\int_{\mathbb{R}^d_+} q \nabla_\alpha \eta h^\alpha dx^\prime dt\bigg|
&\leq C\Big(\int_{\mathbb{R}^d_+}t|\nabla^2 \eta|^2|h|^2dx^\prime dt\Big)^{\frac{1}{2}} \\
& \qquad \times \Bigg\{
\Big(\int_{\mathbb{R}^d_+}t|q|^2dx^\prime dt\Big)^{\frac{1}{2}}
+ \Big(\int_{\mathbb{R}^d_+}t^3|\nabla q|^2dx^\prime dt\Big)^{\frac{1}{2}}
\Bigg\} \\
&+C\big\|\nabla \eta\big\|_{L^\infty(\mathbb{R}^{d-1})}
\Big(\int_{\mathbb{R}^d_+}t|\nabla h|^2dx^\prime dt\Big)^{\frac{1}{2}} \\
&\qquad \times \Bigg\{\Big(\int_{\mathbb{R}^d_+}t|q|^2dx^\prime dt\Big)^{\frac{1}{2}}
+ \Big(\int_{\mathbb{R}^d_+}t^3|\nabla q|^2dx^\prime dt\Big)^{\frac{1}{2}}
\Bigg\}.
\end{aligned}
\end{equation}
In view of Theorem $\ref{thm:2.1}$ and Remark $\ref{re:2.1}$, we have the following estimates
\begin{equation}\label{f:3.2}
\Big(\int_{\mathbb{R}^d_+}t|q|^2dx^\prime dt\Big)^{\frac{1}{2}}
+ \Big(\int_{\mathbb{R}^d_+}t^3|\nabla q|^2dx^\prime dt\Big)^{\frac{1}{2}}
\leq C\Big(\int_{\mathbb{R}^{d-1}}|u|^2 dx^\prime\Big)^{\frac{1}{2}},
\end{equation}
and
\begin{equation}\label{f:3.3}
\Big(\int_{\mathbb{R}^d_+}t|\nabla h|^2dx^\prime dt\Big)^{\frac{1}{2}}
\leq C\Big(\int_{\mathbb{R}^{d-1}}|(h)^*|^2 dx^\prime\Big)^{\frac{1}{2}}.
\end{equation}
Also, it follows from Lemma $\ref{lemma:2.1}$ and \cite[Corollary 7.3.6]{JD} that
\begin{equation}\label{f:3.4}
\Big(\int_{\mathbb{R}^d_+}t|\nabla^2 \eta|^2|h|^2dx^\prime dt\Big)^{\frac{1}{2}}
\leq C\|\eta\|_{C^{0,1}(\mathbb{R}^{d-1})}\Big(\int_{\mathbb{R}^{d-1}}|(h)^*|^2 dS\Big)^{\frac{1}{2}},
\end{equation}
since $|\nabla^2g|tdx^\prime dt$ is the Carleson measure. Consequently,
the desired estimate $\eqref{pri:3.1}$ follows from $\eqref{f:3.1}$,
$\eqref{f:3.2}$, $\eqref{f:3.3}$ and $\eqref{f:3.4}$, and we have completed the proof.
\end{proof}

\begin{remark}
\emph{Although in such the special case $\mathbb{R}^d_+$, the quantity $|\nabla^2\eta|$ may be vanish directly, keeping the term $|\nabla^2\eta|tdx^\prime dt$ suggests where the Carleson measure would be born,
which helps the reader to follow the calculations in later section, easily.}
\end{remark}

\begin{lemma}[Dahlberg's bilinear estimate I]
Let $(h,\pi)\in
H^1(\mathbb{R}_+^d;\mathbb{R}^d)\times L^2(\mathbb{R}_+^d)$ be the solution of
$\Delta h = \nabla \pi$ and $\emph{div}(h) = 0$ in $\mathbb{R}^d_+$, and
$(h)^*\in L^2(\mathbb{R}^{d-1})$. Assume $v=(v_j^\alpha)\in H^1(\mathbb{R}^d_+;\mathbb{R}^{d\times d})$
is supported in $B(0,r_0)$ such that
$v=0$ outside $B(0,r_0)\cap\mathbb{R}_+^d$, where $r_0>0$ is sufficiently large. Then we have
\begin{equation}\label{pri:2.3}
\begin{aligned}
\int_{\mathbb{R}_+^d} \nabla h \cdot v dx
&= \sum_{\alpha =1}^d\sum_{i=1}^{d-1}\int_{\mathbb{R}_+^d}  t\bigg\{\frac{\partial h^\alpha}{\partial t}\frac{\partial v_i^\alpha}{\partial x_i}-\frac{\partial h^\alpha}{\partial x_i}\frac{\partial v_i^\alpha}{\partial t}
-\frac{\partial h^\alpha}{\partial t}\frac{\partial v_d^\alpha}{\partial t}\bigg\}dx^\prime dt\\
& -\sum_{\beta =1}^{d-1}\sum_{i=1}^{d-1}\int_{\mathbb{R}_+^d} t \bigg\{\frac{\partial h^\beta}{\partial x_i}
\frac{\partial v_d^\beta}{\partial x_i}
+ \frac{\partial h^\beta}{\partial t}
\frac{\partial v_d^d}{\partial x_\beta}
-\pi \frac{\partial v_d^\beta}{\partial x_\beta}
\bigg\}
dx^\prime dt,
\end{aligned}
\end{equation}
where $dx=dx^\prime dt$. Moreover, there admits the following estimate
\begin{equation}\label{pri:3.2}
\bigg|\int_{\mathbb{R}_+^d} \nabla h \cdot v dx\bigg|
\leq C\Big(\int_{\mathbb{R}_+^d} |\nabla v|^2 t dx^\prime dt\Big)^{\frac{1}{2}}
\Bigg\{ \int_{\mathbb{R}_+^d} |\nabla h|^2 t dx^\prime dt\Big)^{\frac{1}{2}}
+ \int_{\mathbb{R}_+^d} |\pi|^2 t dx^\prime dt\Big)^{\frac{1}{2}}
\Bigg\}
\end{equation}
where $C$ depends only on $d$.
\end{lemma}

\begin{proof}
For ease of statement, let $i,\beta = 1,\cdots,d-1$, and $j,\alpha = 1,\cdots,d$, and the summation convention for repeated indices will be used. We divide the left-hand side of $\eqref{pri:2.3}$ into
two parts as follows:
\begin{equation}\label{f:2.10}
\int_{\mathbb{R}^d_+} \frac{\partial h^\alpha}{\partial x_j} v_j^\alpha dx^\prime dt
= \int_{\mathbb{R}^d_+} \frac{\partial h^\alpha}{\partial x_i} v_i^\alpha dx^\prime dt
+ \int_{\mathbb{R}^d_+} \frac{\partial h^\alpha}{\partial t} v_d^\alpha dx^\prime dt
= I_1 + I_2.
\end{equation}
The simple part is $I_1$, and taking integration by parts with respect to $t$, we have
\begin{equation}\label{f:2.9}
\begin{aligned}
I_1 &= -\int_{\mathbb{R}^d_+}t\frac{\partial}{\partial t}\bigg\{\frac{\partial h^\alpha}{\partial x_i}
v_i^\alpha\bigg\} dx^\prime dt
= -\int_{\mathbb{R}^d_+}t\frac{\partial^2 h^\alpha}{\partial x_i\partial t}
v_i^\alpha dx^\prime dt
-\int_{\mathbb{R}^d_+}t\frac{\partial h^\alpha}{\partial x_i}
\frac{\partial v_i^\alpha}{\partial t} dx^\prime dt\\
& = \int_{\mathbb{R}^d_+}t\frac{\partial h^\alpha}{\partial t}
\frac{\partial v_i^\alpha}{\partial x_i} dx^\prime dt
-\int_{\mathbb{R}^d_+}t\frac{\partial h^\alpha}{\partial x_i}
\frac{\partial v_i^\alpha}{\partial t} dx^\prime dt.
\end{aligned}
\end{equation}
To handle $I_2$, we write
\begin{equation*}
\begin{aligned}
I_2 &= \int_{\mathbb{R}^d_+} \frac{\partial h^\beta}{\partial t} v_d^\beta dx^\prime dt
+ \int_{\mathbb{R}^d_+} \frac{\partial h^d}{\partial t} v_d^d dx^\prime dt \\
& = -\int_{\mathbb{R}^d_+} t\frac{\partial}{\partial t}\bigg\{\frac{\partial h^\beta}{\partial t} v_d^\beta \bigg\} dx^\prime dt
- \int_{\mathbb{R}^d_+} \frac{\partial h^\beta}{\partial x_\beta} v_d^d dx^\prime dt
= I_{21} + I_{22}.
\end{aligned}
\end{equation*}
We mention that the second equality above follows from the fact that $\text{div}(h)=0$ in $\mathbb{R}_+^d$. For $I_{21}$, we indeed employ $\Delta h =\nabla\pi$ in $\mathbb{R}_+^d$
to compute the integral
\begin{equation*}
\begin{aligned}
\int_{\mathbb{R}^d_+} t\frac{\partial^2 h^\beta}{\partial t^2} v^\beta_d dx^\prime dt
&= - \int_{\mathbb{R}^d_+} t\frac{\partial^2 h^\beta}{\partial x_i^2} v^\beta_d dx^\prime dt
+ \int_{\mathbb{R}^d_+} t \nabla_\beta\pi v_d^\beta dx^\prime dt \\
& = \int_{\mathbb{R}^d_+} t\frac{\partial h^\beta}{\partial x_i} \frac{v^\beta_d}{\partial x_i}
dx^\prime dt
- \int_{\mathbb{R}^d_+} t \pi \frac{\partial v_d^\beta}{\partial x_\beta} dx^\prime dt
\end{aligned}
\end{equation*}
where we use the integration by parts with respect to $x_i$ in the second equality.
Note that all of $\frac{\partial}{\partial x_i}$ and $\frac{\partial}{\partial x_\beta}$
are tangential derivative. The core idea is that using tangential derivatives control the conormal derivative. Thus, we have
\begin{equation}\label{f:2.7}
I_{21} = -\int_{\mathbb{R}^d_+} t\frac{\partial h^\beta}{\partial x_i} \frac{\partial v^\beta_d}{\partial x_i}
dx^\prime dt
-\int_{\mathbb{R}^d_+} t\frac{\partial h^\beta}{\partial t} \frac{\partial v^\beta_d}{\partial t}
dx^\prime dt
+ \int_{\mathbb{R}^d_+} t \pi \frac{\partial v_d^\beta}{\partial x_\beta} dx^\prime dt.
\end{equation}
Proceeding as in the proof of $I_1$, we have
\begin{equation}\label{f:2.8}
\begin{aligned}
I_{22} &= \int_{\mathbb{R}^d_+} t\frac{\partial}{\partial t}\bigg\{\frac{\partial h^\beta}{\partial x_\beta} v_d^d \bigg\}dx^\prime dt
& = -\int_{\mathbb{R}^d_+} t\frac{\partial h^\beta}{\partial t} \frac{\partial v_d^d}{\partial x_\beta}dx^\prime dt
+ \int_{\mathbb{R}^d_+} t\frac{\partial h^\beta}{\partial x_\beta} \frac{\partial v_d^d}{\partial t}dx^\prime dt.
\end{aligned}
\end{equation}
Combining equalities $\eqref{f:2.7}$ and $\eqref{f:2.8}$ leads to
\begin{equation*}
I_2 = -\int_{\mathbb{R}^d_+} t\frac{\partial h^\beta}{\partial x_i} \frac{v^\beta_d}{\partial x_i}
dx^\prime dt
-\int_{\mathbb{R}^d_+} t\frac{\partial h^\beta}{\partial t} \frac{\partial v_d^d}{\partial x_\beta}dx^\prime dt
+ \int_{\mathbb{R}^d_+} t \pi \frac{\partial v_d^\beta}{\partial x_\beta} dx^\prime dt
-\int_{\mathbb{R}^d_+} t\frac{\partial h^\alpha}{\partial t} \frac{\partial v^\alpha_d}{\partial t}
dx^\prime dt
\end{equation*}
where we use the fact that $\text{div}(h)=0$ in $\mathbb{R}_+^d$, again. This together with
$\eqref{f:2.9}$ and $\eqref{f:2.10}$ gives the desired identity $\eqref{pri:2.3}$.

Then we take the last term in the right-hand side of identity $\eqref{pri:2.3}$ as an example:
\begin{equation*}
\bigg|\int_{\mathbb{R}^d_+} t \pi \frac{\partial v_d^\beta}{\partial x_\beta} dx^\prime dt\bigg|
\leq \int_{\mathbb{R}^d_+} t^{\frac{1}{2}}|\pi| |\nabla v|t^{\frac{1}{2}} dx^\prime dt
\leq \Big(\int_{\mathbb{R}^d_+} |\pi|^2t dx^\prime dt\Big)^{\frac{1}{2}}
\Big(\int_{\mathbb{R}^d_+} |\nabla v|^2t dx^\prime dt\Big)^{\frac{1}{2}}
\end{equation*}
where we employ Cauchy's inequality in the last step.
The desired estimate $\eqref{pri:3.2}$ simply follows from the same manner and
we have completed the proof.
\end{proof}

\begin{flushleft}
\textbf{Proof of Theorem \ref{thm:3.1}.}
In view of the identity $\eqref{pri:2.1}$, to estimate the quantity
$\|\Lambda(\eta f)-\eta\Lambda(f)\|_{L^2(\mathbb{R}^{d-1})}$, it is reduced to control
\end{flushleft}
\vspace{-0.2cm}
\begin{equation}\label{f:3.5}
\bigg|\int_{\mathbb{R}^d_+} u^\alpha \nabla \eta \nabla h^\alpha dx\bigg|
+ \bigg|\int_{\mathbb{R}^d_+} \nabla u^\alpha \nabla \eta h^\alpha dx\bigg|
+ \bigg|\int_{\mathbb{R}^d_+} q\nabla_\alpha \eta  h^\alpha dx\bigg|
+ \bigg|\int_{\mathbb{R}^d_+} \pi\nabla_\alpha \eta u^\alpha dx\bigg|.
\end{equation}

For the first term in $\eqref{f:3.5}$, choose $v_i^\alpha = u^\alpha\nabla_i \eta$ with
$i,\alpha=1,\cdots,d$, and it follows from the Dahlberg's bilinear estimate $\eqref{pri:3.2}$ that
\begin{equation}\label{f:3.6}
\begin{aligned}
\bigg|\int_{\mathbb{R}^d_+} u^\alpha \nabla \eta \nabla h^\alpha dx\bigg|
&\leq C\Bigg\{\Big(\int_{\mathbb{R}_+^d} |u|^2|\nabla^2 \eta|^2 t dx^\prime dt\Big)^{\frac{1}{2}}
+\|\eta\|_{C^{0,1}(\mathbb{R}^{d-1})}\Big(\int_{\mathbb{R}_+^d} |\nabla u|^2 t dx^\prime dt\Big)^{\frac{1}{2}}
\Bigg\}\\
&\qquad \times\Bigg\{ \int_{\mathbb{R}_+^d} |\nabla h|^2 t dx^\prime dt\Big)^{\frac{1}{2}}
+ \int_{\mathbb{R}_+^d} |\pi|^2 t dx^\prime dt\Big)^{\frac{1}{2}}
\Bigg\} \\
&\leq C\big\|\eta\big\|_{C^{0,1}(\mathbb{R}^{d-1})}\Big(\int_{\mathbb{R}^{d-1}} |(u)^*|^2 dx^\prime\Big)^{\frac{1}{2}}
\Bigg\{ \int_{\mathbb{R}_+^d} |(h)^*|^2 dx^\prime dt\Big)^{\frac{1}{2}}
+ \int_{\mathbb{R}_+^d} |h|^2 dx^\prime dt\Big)^{\frac{1}{2}}
\Bigg\}\\
&\leq C\big\|\eta\big\|_{C^{0,1}(\mathbb{R}^{d-1})}\big\|f\big\|_{L^2(\mathbb{R}^{d-1})}\big\|h\big\|_{L^2(\mathbb{R}^{d-1})}.
\end{aligned}
\end{equation}
In the second inequality, we employ the estimates $\eqref{pri:2.4}$ and $\eqref{pri:2.5}$, as well as
Lemma $\ref{lemma:2.1}$ coupled with \cite[Corollary 7.3.6]{JD}. In the last one, we use the nontangential maximal
function estimates
$\|(u)^*\|_{L^2(\mathbb{R}^{d-1})}\leq C\|f\|_{L^2(\mathbb{R}^{d-1})}$ and
$\|(h)^*\|_{L^2(\mathbb{R}^{d-1})}\leq C\|h\|_{L^2(\mathbb{R}^{d-1})}$ (see \cite[Theorem 3.9]{EBFCEKGCV}).

The second one in $\eqref{f:3.5}$ obeys the same procedure. It suffices to choose
$v_i^\alpha = \nabla_i \eta h^\alpha$, and it is not hard to see that
\begin{equation}\label{f:3.7}
\begin{aligned}
\bigg|\int_{\mathbb{R}^d_+} \nabla u^\alpha \nabla \eta h^\alpha dx\bigg|
&\leq C\Bigg\{\Big(\int_{\mathbb{R}_+^d} |h|^2|\nabla^2 \eta|^2 t dx^\prime dt\Big)^{\frac{1}{2}}
+\|\eta\|_{C^{0,1}(\mathbb{R}^{d-1})}\Big(\int_{\mathbb{R}_+^d} |\nabla h|^2 t dx^\prime dt\Big)^{\frac{1}{2}}
\Bigg\}\\
&\qquad \times\Bigg\{ \int_{\mathbb{R}_+^d} |\nabla u|^2 t dx^\prime dt\Big)^{\frac{1}{2}}
+ \int_{\mathbb{R}_+^d} |q|^2 t dx^\prime dt\Big)^{\frac{1}{2}}
\Bigg\} \\
&\leq C\|\eta\|_{C^{0,1}(\mathbb{R}^{d-1})}\Big(\int_{\mathbb{R}^{d-1}} |(h)^*|^2 dx^\prime\Big)^{\frac{1}{2}}
\Bigg\{ \int_{\mathbb{R}_+^d} |(u)^*|^2 dx^\prime dt\Big)^{\frac{1}{2}}
+ \int_{\mathbb{R}_+^d} |u|^2 dx^\prime dt\Big)^{\frac{1}{2}}
\Bigg\}\\
&\leq C\big\|\eta\big\|_{C^{0,1}(\mathbb{R}^{d-1})}
\big\|f\big\|_{L^2(\mathbb{R}^{d-1})}\big\|h\big\|_{L^2(\mathbb{R}^{d-1})}.
\end{aligned}
\end{equation}

Proceeding as in the proof of the estimate $\eqref{pri:3.1}$, we obtain
\begin{equation}\label{f:3.8}
\begin{aligned}
\bigg|\int_{\mathbb{R}^d_+} \pi \nabla_\alpha \eta u^\alpha dx^\prime dt\bigg|
&\leq C\|\eta\|_{C^{0,1}(\mathbb{R}^{d-1})}\|(u)^*\|_{L^2(\mathbb{R}^{d-1})}\|h\|_{L^2(\mathbb{R}^{d-1})}\\
&\leq C\|\eta\|_{C^{0,1}(\mathbb{R}^{d-1})}\|f\|_{L^2(\mathbb{R}^{d-1})}\|h\|_{L^2(\mathbb{R}^{d-1})}.
\end{aligned}
\end{equation}
Note that $(u,q)$ satisfies the equation $\eqref{pde:1.1}$. Thus, plugging the estimates
$\eqref{f:3.6}$, $\eqref{f:3.6}$, $\eqref{f:3.1}$ and $\eqref{f:3.6}$ back into $\eqref{f:3.5}$ leads
to the desired estimate $\eqref{pri:3.3}$, and we have completed the proof.
\qed

\section{The proof of Theorem $\ref{thm:1.1}$}

\begin{lemma}[Dahlberg's bilinear estimate II]\label{lemma:4.1}
Let $(h,\pi)\in
H^1(\Omega;\mathbb{R}^d)\times L^2(\Omega)$ be the solution of
$\Delta h = \nabla \pi$ and $\emph{div}(h) = 0$ in $\Omega$, and
$(h)^*\in L^2(\partial\Omega)$. Assume
Then for any  $v=(v_j^\alpha)\in H^1(\Omega;\mathbb{R}^{d\times d})$, we have
\begin{equation}\label{pri:4.1}
\begin{aligned}
\bigg|\int_{\Omega} \nabla h \cdot v dx\bigg|
&\leq C
\Bigg\{ \Big(\int_{\Omega} |\nabla h|^2 \delta(x) dx\Big)^{\frac{1}{2}}
+ \Big(\int_{\Omega} |\pi|^2 \delta(x) dx\Big)^{\frac{1}{2}}
\Bigg\} \\
&\qquad\qquad\qquad\times
\Bigg\{\Big(\int_{\Omega} |\nabla v|^2 \delta(x) dx\Big)^{\frac{1}{2}}
+
\Big(\int_{\partial\Omega} |(v)^*|^2 dS\Big)^{\frac{1}{2}}
\Bigg\}
\end{aligned}
\end{equation}
where $\delta(x)=\emph{dist}(x,\partial\Omega)$, and $C$ depends only on $d$ and $\Omega$.
\end{lemma}

\begin{proof}
By linear transformation both in the variable $x$ and the solution $(h,\pi)$, we may assume that
\begin{equation*}
 D_r=\Omega\cap B(P,r) = \big\{(x^\prime,y)\in\mathbb{R}^{d}:y>\psi(x^\prime)\big\}\cap B(P,r).
\end{equation*}
where $\psi$ is a Lipschitz function on $\mathbb{R}^{d-1}$. Let $\eta\in C_0^\infty(B(P,2r))$
be a cut-off function such that $\eta =1$ in $B(P,r)$. Since $\Delta h = \nabla\pi$, and
$\text{div}(h)=0$ in $\Omega$, it is not hard to derive
\begin{equation*}
\Delta(\eta h) = \nabla(\eta\pi) - \tilde{f},
\quad \text{and}\quad \text{div}(\eta h) = h\cdot\nabla\eta
\qquad \text{in}~\Omega,
\end{equation*}
where $\tilde{f} = \pi\nabla\eta-2\nabla h\cdot\nabla\eta-h\Delta\eta$.
Thus it is enough to establish $\eqref{pri:4.1}$ with $\Omega$ replaced by $D_r$, assuming that $\Delta h = \nabla\pi$ and $\text{div}(h) = 0$ in $D_r$, and $v\in H^1_0(B(0,r))$.
Furthermore, since the Carleson measure is translation and rotation invariant, it is fine to assume
$P=0$.

By a special change of variables invented by C. Kenig and E. Stein, we may further reduce the problem
to the case of upper half-space $\mathbb{R}^d_+$. Indeed, let $\rho:\mathbb{R}_+^d\to D=D_r$ be defined
by
\begin{equation}\label{eq:4.1}
\rho(x^\prime,t) = (x^\prime,y) = (x^\prime,\varphi(x,t)) = (x^\prime,c_0t+\zeta_t\ast\psi(x^\prime)),
\end{equation}
where $\zeta_t(x^\prime) = t^{1-d}\zeta(x^\prime/t)$ is a smooth compactly supported bump function and
the constant $c_0 = c_0(d,\|\nabla\psi\|_{L^\infty(\mathbb{R}^{d-1})})$ is sufficient large such that
$\frac{\partial\varphi}{\partial t}\geq \frac{1}{8}$. The map
$\rho$ is a bi-Lipschitz map, which owns two essential properties: (1)
there exists two constant $c,C>0$ such that $c\leq |\nabla\rho(x^\prime,t)|\leq C$; and (2)
$|\nabla^2\varphi(x^\prime,t)|^2tdx^\prime dt$ (or in another form
$|\nabla^2\rho(x^\prime,t)|^2tdx^\prime dt$) is a Carleson measure on $\mathbb{R}_+^d$; Hence, the
property (1) and
\begin{equation*}
\int_{D} \nabla h \cdot v dx = \int_{\mathbb{R}^{d}_+}\nabla h\circ\rho \cdot v\circ\rho |\nabla\rho| dx^\prime dt
\end{equation*}
indicate that it suffices to show
\begin{equation}\label{f:4.1}
\begin{aligned}
\bigg|\int_{\mathbb{R}^d_+} \nabla h\circ\rho \cdot v\circ\rho |\nabla\rho| dx^\prime dt\bigg|
&\leq C
\Bigg\{ \Big(\int_{\mathbb{R}^d_+} |\nabla h\circ\rho|^2 t dx^\prime dt\Big)^{\frac{1}{2}}
+ \int_{\mathbb{R}^d_+} |\pi\circ\rho|^2 t dx^\prime dt\Big)^{\frac{1}{2}}
\Bigg\} \\
&\times
\Bigg\{\Big(\int_{\mathbb{R}^d_+} |\nabla v\circ\rho|^2 t dx^\prime dt\Big)^{\frac{1}{2}}
+
\Big(\int_{\mathbb{R}^{d-1}} |(v\circ\rho)^*|^2 dx^\prime\Big)^{\frac{1}{2}}
\Bigg\}.
\end{aligned}
\end{equation}

The remainder of the argument is analogous to that in Theorem $\ref{thm:3.1}$, and we only focus on the
different places. Let $1\leq i,k\leq d-1$ and $1\leq \alpha\leq d$.
and we first divide the integral in the left-hand side of $\eqref{f:4.1}$ into
two parts
\begin{equation}
\begin{aligned}
\int_{\mathbb{R}^d_+} \nabla h\circ\rho \cdot v\circ\rho |\nabla\rho| dx^\prime dt
&= \int_{\mathbb{R}^d_+} \frac{\partial h^\alpha}{\partial x_i}\circ\rho \cdot v_i^\alpha\circ\rho |\nabla\rho| dx^\prime dt
+ \int_{\mathbb{R}^d_+} \frac{\partial h^\alpha}{\partial y}\circ\rho \cdot v_d^\alpha\circ\rho |\nabla\rho| dx^\prime dt\\
& := A_1 + A_2.
\end{aligned}
\end{equation}
Unlike the proof of Theorem $\ref{thm:3.1}$, the difficulty has already appeared in calculating $A_1$,
and it follows that
\begin{equation}
\begin{aligned}
A_1 & = -\int_{\mathbb{R}^d_+} t\frac{\partial^2 h^\alpha}{\partial x_i\partial y}\circ\rho\frac{\partial\varphi}{\partial t}\cdot v_i^\alpha\circ\rho |\nabla\rho| dx^\prime dt
- \int_{\mathbb{R}^d_+} t\frac{\partial h^\alpha}{\partial x_i}\circ\rho\cdot \frac{\partial v_i^\alpha}{\partial y}\circ\rho\frac{\partial\varphi}{\partial t} |\nabla\rho| dx^\prime dt \\
&\qquad\qquad-\int_{\mathbb{R}^d_+} t\frac{\partial h^\alpha}{\partial x_i}\circ\rho\cdot v_i^\alpha\circ\rho \frac{\nabla\rho}{|\nabla\rho|}\cdot\frac{\partial}{\partial t}\big(\nabla\rho\big) dx^\prime dt
:= B_1 + B_2 + B_3
\end{aligned}
\end{equation}
Noting that $B_2$ is a good term, we have
\begin{equation}\label{f:4.11}
|B_2|\leq C\int_{\mathbb{R}^d_+} t\Big|\frac{\partial h^\alpha}{\partial x_i}\circ\rho\Big|\cdot \Big|\frac{\partial v_i^\alpha}{\partial y}\circ\rho\Big| dx^\prime dt
\leq C
\Big(\int_{\mathbb{R}^d_+} |\nabla h\circ\rho|^2 t dx^\prime dt\Big)^{\frac{1}{2}}
\Big(\int_{\mathbb{R}^d_+} |\nabla v\circ\rho|^2 t dx^\prime dt\Big)^{\frac{1}{2}}.
\end{equation}
Before studying $A_1$, we point out that $|A_3|$ will produce a Carleson measure
$|\nabla^2\rho|^2tdx^\prime dt$, and we will see that
\begin{equation}\label{f:4.12}
\begin{aligned}
|B_3| &\leq C\Big(\int_{\mathbb{R}^d_+} |\nabla h\circ\rho|^2 t dx^\prime dt\Big)^{\frac{1}{2}}
\Big(\int_{\mathbb{R}^d_+} |v\circ\rho|^2 |\nabla^2\rho|t dx^\prime dt\Big)^{\frac{1}{2}} \\
&\leq C\Big(\int_{\mathbb{R}^d_+} |\nabla h\circ\rho|^2 t dx^\prime dt\Big)^{\frac{1}{2}}
\Big(\int_{\mathbb{R}^{d-1}} |(v\circ\rho)^*|^2 dx^\prime\Big)^{\frac{1}{2}},
\end{aligned}
\end{equation}
where we employ \cite[Corollary 7.3.6]{JD} in the last inequality. In this sense, the factor
$|\nabla\rho|$ (or $\nabla\varphi$) is good in the left-hand side of $\eqref{f:4.1}$, which actually may
produce a Carleson measure in the integral.

To estimate $|B_1|$, we need to move the derivative $\partial/\partial x_i$ of
$\frac{\partial^2 h^\alpha}{\partial x_i\partial y}$ to other terms through integration by parts.
Plugging the following identity
\begin{equation*}
\frac{\partial^2 h^\alpha}{\partial x_i\partial y}\circ\rho
= \frac{\partial}{\partial x_i}\Big\{\frac{\partial h^\alpha}{\partial y}\circ\rho\Big\}
- \frac{\partial^2 h^\alpha}{\partial y^2}\circ\rho\frac{\partial\varphi}{\partial x_i}
\end{equation*}
back into $B_1$, the second term in the right-hand side above will bring the real difficulty, and
we merely study
\begin{equation}\label{f:4.2}
\int_{\mathbb{R}^d_+} t\frac{\partial^2 h^\alpha}{\partial y^2}\circ\rho
\frac{\partial\varphi}{\partial x_i}
\frac{\partial\varphi}{\partial t}
\cdot v_i^\alpha\circ\rho |\nabla\rho| dx^\prime dt := E_1.
\end{equation}
Since $\Delta h^\alpha = \nabla_\alpha$ in $D$,
it is not hard to derive that
\begin{equation*}
\begin{aligned}
\frac{\partial^2 h^\alpha}{\partial y^2}\circ\rho
&= - \sum_{k=1}^{d-1}\frac{\partial^2 h^\alpha}{\partial x_k^2}\circ\rho + \nabla_\alpha\pi\circ\rho \\
& = -\frac{\partial}{\partial x_k}\Big\{\frac{\partial h^\alpha}{\partial x_k}\circ\rho\Big\}
+ \frac{\partial^2 h^\alpha}{\partial x_k\partial y}\circ\rho\frac{\partial\varphi}{\partial x_k}
+ \nabla_\alpha\pi\circ\rho \\
& = -\frac{\partial}{\partial x_k}\Big\{\frac{\partial h^\alpha}{\partial x_k}\circ\rho\Big\}
+ \frac{\partial}{\partial x_k}\Big\{\frac{\partial h^\alpha}{\partial y}\circ\rho\Big\}
\frac{\partial\varphi}{\partial x_k}
- \frac{\partial^2 h^\alpha}{\partial y^2}\circ\rho |\nabla_{x^\prime}\varphi|^2
+ \nabla_\alpha\pi\circ\rho,
\end{aligned}
\end{equation*}
where $|\nabla_{x^\prime}\varphi|^2 = \sum_{k=1}^{d-1}(\partial\varphi/\partial x_k)^2$
and $\nabla_{x^\prime} = (\partial_1,\cdots,\partial_{d-1})$, and this implies
\begin{equation}\label{f:4.3}
\big(1+|\nabla_{x^\prime}\varphi|^2\big)\frac{\partial^2 h^\alpha}{\partial y^2}\circ\rho
= -\frac{\partial}{\partial x_k}\Big\{\frac{\partial h^\alpha}{\partial x_k}\circ\rho\Big\}
+ \frac{\partial}{\partial x_k}\Big\{\frac{\partial h^\alpha}{\partial y}\circ\rho\Big\}
\frac{\partial\varphi}{\partial x_k}
+ \nabla_\alpha\pi\circ\rho.
\end{equation}
Then inserting the identity $\eqref{f:4.3}$ into $\eqref{f:4.2}$, we have
\begin{equation}\label{f:4.4}
\begin{aligned}
E_1 &= -\int_{\mathbb{R}^d_+} t
\frac{\partial}{\partial x_k}\Big\{\frac{\partial h^\alpha}{\partial x_k}\circ\rho\Big\}
\cdot v_i^\alpha\circ\rho
\frac{\partial\varphi}{\partial x_i}
\frac{\big(\frac{\partial\varphi}{\partial t}\big)^2}{1+|\nabla_{x^\prime}\varphi|} dx^\prime dt \\
& \quad + \int_{\mathbb{R}^d_+} t
\frac{\partial}{\partial x_k}\Big\{\frac{\partial h^\alpha}{\partial y}\circ\rho\Big\}
\cdot v_i^\alpha\circ\rho
\frac{\partial\varphi}{\partial x_k}
\frac{\partial\varphi}{\partial x_i}
\frac{\big(\frac{\partial\varphi}{\partial t}\big)^2}{1+|\nabla_{x^\prime}\varphi|} dx^\prime dt\\
&\quad + \int_{\mathbb{R}^d_+} t
\frac{\partial \pi}{\partial x_\alpha}\circ\rho
\cdot v_i^\alpha\circ\rho
\frac{\partial\varphi}{\partial x_i}
\frac{\big(\frac{\partial\varphi}{\partial t}\big)^2}{1+|\nabla_{x^\prime}\varphi|} dx^\prime dt
\end{aligned}
\end{equation}
where we use the fact that $|\nabla \rho| = \frac{\partial\varphi}{\partial t}$, and we denote
the last term in the right-hand side of $\eqref{f:4.4}$ by $E_2$. Then,
for the first two terms in the right-hand side of $\eqref{f:4.4}$, proceeding as in the
proof for $|B_3|$, we state the following result without details,
\begin{equation}\label{f:4.9}
|E_1-E_2|\leq
C\Big(\int_{\mathbb{R}^d_+} |\nabla h\circ\rho|^2 t dx^\prime dt\Big)^{\frac{1}{2}}
\Bigg\{\Big(\int_{\mathbb{R}^d_+} |\nabla v\circ\rho|^2 t dx^\prime dt\Big)^{\frac{1}{2}} +
\Big(\int_{\mathbb{R}^{d-1}} |(v\circ\rho)^*|^2 dx^\prime\Big)^{\frac{1}{2}}\Bigg\}.
\end{equation}

It is time to handle the problem brought by the derivative of the pressure term in $E_2$.
In fact, the bad case is just related to the factor $\frac{\partial\pi}{\partial y}$, since we have
\begin{equation}\label{f:4.5}
\begin{aligned}
E_2 &=  \sum_{k=1}^{d-1}\int_{\mathbb{R}^d_+} t\Bigg\{
\frac{\partial }{\partial x_k}\big\{\pi\circ\rho\big\}-\frac{\partial\pi}{\partial y}\circ\rho
\frac{\partial\varphi}{\partial x_k}\Bigg\}
\cdot v_i^k\circ\rho
\frac{\partial\varphi}{\partial x_i}
\frac{\big(\frac{\partial\varphi}{\partial t}\big)^2}{1+|\nabla_{x^\prime}\varphi|} dx^\prime dt \\
& \quad + \int_{\mathbb{R}^d_+} t\frac{\partial\pi}{\partial y}\circ\rho
\cdot v_i^d\circ\rho
\frac{\partial\varphi}{\partial x_i}
\frac{\big(\frac{\partial\varphi}{\partial t}\big)^2}{1+|\nabla_{x^\prime}\varphi|} dx^\prime dt.
\end{aligned}
\end{equation}

Hence, the problem is reduced to estimate the second line of $\eqref{f:4.5}$, denoted by $E_3$.
Note the fact that $\Delta\pi = 0$ in $D$ (see \cite[pp.1204]{RS} or \cite[pp.773]{EBFCEKGCV}),
some tedious manipulation yields
\begin{equation}\label{f:4.6}
\big(1+|\nabla_{x^\prime}\varphi|^2\big)\frac{\partial^2\pi}{\partial y^2}\circ\rho
= -\frac{\partial}{\partial x_k}\Big\{\frac{\partial \pi}{\partial x_k}\circ\rho\Big\}
+ \frac{\partial}{\partial x_k}\Big\{\frac{\partial\pi}{\partial y}\circ\rho\Big\}
\frac{\partial\varphi}{\partial x_k}.
\end{equation}
Moreover, taking integration by parts with respect to $t$ in $E_3$, we obtain
\begin{equation}
\begin{aligned}
E_3 & = - \int_{\mathbb{R}^d_+} t^2\frac{\partial^2\pi}{\partial y^2}\circ\rho
\cdot v_i^d\circ\rho
\frac{\partial\varphi}{\partial x_i}
\frac{\big(\frac{\partial\varphi}{\partial t}\big)^3}{1+|\nabla_{x^\prime}\varphi|} dx^\prime dt \\
& \quad - \int_{\mathbb{R}^d_+} t^2\frac{\partial\pi}{\partial y}\circ\rho
\cdot \frac{\partial}{\partial t}\Bigg\{v_i^d\circ\rho
\frac{\partial\varphi}{\partial x_i}
\frac{\big(\frac{\partial\varphi}{\partial t}\big)^2}{1+|\nabla_{x^\prime}\varphi|} \Bigg\}dx^\prime dt
:= E_3^1 + E_2^2.
\end{aligned}
\end{equation}
Obviously, the term $E_3^2$ may produce a Carleson measure, and we first handle it. By Cauchy's inequality, it follows that
\begin{equation}\label{f:4.7}
\begin{aligned}
|E_3^2| &\leq C\Bigg\{\int_{\mathbb{R}^d_+} t^2|\nabla\pi\circ\rho||\nabla v\circ\rho|dx^\prime dt
+\int_{\mathbb{R}^d_+} t^2|\nabla\pi\circ\rho||v\circ\rho||\nabla^2\varphi|dx^\prime dt\Bigg\} \\
&\leq C\Big(\int_{\mathbb{R}^d_+} t^3|\nabla\pi\circ\rho|^2dx^\prime dt\Big)^{\frac{1}{2}}
\Bigg\{\Big(\int_{\mathbb{R}^d_+} t|\nabla v\circ\rho|^2dx^\prime dt\Big)^{\frac{1}{2}}
+\Big(\int_{\mathbb{R}^d_+} |v\circ\rho|^2 |\nabla^2\varphi|t dx^\prime dt\Big)^{\frac{1}{2}}\Bigg\}\\
&\leq C\Big(\int_{\mathbb{R}^d_+} t|\pi\circ\rho|^2dx^\prime dt\Big)^{\frac{1}{2}}
\Bigg\{\Big(\int_{\mathbb{R}^d_+} t|\nabla v\circ\rho|^2dx^\prime dt\Big)^{\frac{1}{2}}
+\Big(\int_{\mathbb{R}^{d-1}} |(v\circ\rho)^*|^2 dx^\prime\Big)^{\frac{1}{2}}\Bigg\},
\end{aligned}
\end{equation}
where we use the estimate $\eqref{pri:2.5}$, as well as \cite[Corollary 7.3.6]{JD}  in the last step.
Then we turn to study $E_3^1$, and it follows from the identity $\eqref{f:4.6}$ that
\begin{equation*}
\begin{aligned}
E_3^1 &= \int_{\mathbb{R}^d_+} t^2\frac{\partial}{\partial x_k}\Big\{\frac{\partial \pi}{\partial x_k}\circ\rho\Big\}
\cdot v_i^d\circ\rho
\frac{\partial\varphi}{\partial x_i}
\frac{\big(\frac{\partial\varphi}{\partial t}\big)^3}{1+|\nabla_{x^\prime}\varphi|} dx^\prime dt \\
& \qquad -\int_{\mathbb{R}^d_+} t^2\frac{\partial}{\partial x_k}\Big\{\frac{\partial \pi}{\partial y}\circ\rho\Big\}
\cdot v_i^d\circ\rho
\frac{\partial\varphi}{\partial x_k}
\frac{\partial\varphi}{\partial x_i}
\frac{\big(\frac{\partial\varphi}{\partial t}\big)^3}{1+|\nabla_{x^\prime}\varphi|} dx^\prime dt.
\end{aligned}
\end{equation*}
Integrating by parts in $x_k$, and proceeding as in the proof of $E_3^2$, we also arrive at
\begin{equation}\label{f:4.8}
|E_3^1|\leq C\Big(\int_{\mathbb{R}^d_+} t|\pi\circ\rho|^2dx^\prime dt\Big)^{\frac{1}{2}}
\Bigg\{\Big(\int_{\mathbb{R}^d_+} t|\nabla v\circ\rho|^2dx^\prime dt\Big)^{\frac{1}{2}}
+\Big(\int_{\mathbb{R}^{d-1}} |(v\circ\rho)^*|^2 dx^\prime\Big)^{\frac{1}{2}}\Bigg\}
\end{equation}
Hence, the estimates $\eqref{f:4.7}$ and $\eqref{f:4.8}$ lead to the estimate of $|E_3|$,
and then it is not hard to see
\begin{equation*}
|E_2|\leq C\Big(\int_{\mathbb{R}^d_+} t|\pi\circ\rho|^2dx^\prime dt\Big)^{\frac{1}{2}}
\Bigg\{\Big(\int_{\mathbb{R}^d_+} t|\nabla v\circ\rho|^2dx^\prime dt\Big)^{\frac{1}{2}}
+\Big(\int_{\mathbb{R}^{d-1}} |(v\circ\rho)^*|^2 dx^\prime\Big)^{\frac{1}{2}}\Bigg\},
\end{equation*}
which together with the estimate $\eqref{f:4.9}$ gives
\begin{equation}\label{f:4.10}
\begin{aligned}
|E_1|
&\leq C \Bigg\{ \int_{\mathbb{R}^d_+} |\nabla h\circ\rho|^2 t dx^\prime dt\Big)^{\frac{1}{2}}
+ \int_{\mathbb{R}^d_+} |\pi\circ\rho|^2 t dx^\prime dt\Big)^{\frac{1}{2}}
\Bigg\} \\
&\qquad\qquad\quad\times
\Bigg\{\Big(\int_{\mathbb{R}^d_+} |\nabla v\circ\rho|^2 t dx^\prime dt\Big)^{\frac{1}{2}}
+\Big(\int_{\mathbb{R}^{d-1}} |(v\circ\rho)^*|^2 dx^\prime\Big)^{\frac{1}{2}}
\Bigg\}.
\end{aligned}
\end{equation}

Recalling the expression of $E_1$, we may follow the same procedure
above to estimate $|A_2|$, and the details are left to the reader.
Up to now, it is not hard to verify that $|B_1|$ is controlled by the right-hand side of
$\eqref{f:4.10}$ with a different constant $C$. By noting the estimates $\eqref{f:4.11}$ and $\eqref{f:4.12}$,
we have indeed proved the estimate $\eqref{f:4.1}$, and the proof is complete.
\end{proof}

\begin{lemma}\label{lemma:4.2}
Let $(u,q)\in
H^1(\Omega;\mathbb{R}^d)\times L^2(\Omega)$ be the solution of
$\Delta u = \nabla q$ and $\emph{div}(u) = 0$ in $\Omega$, and
$(u)^*\in L^2(\partial\Omega)$. Assume $\eta\in C^{0,1}(\partial\Omega)$, and
the vector-valued function $h$ is given as in Lemma $\ref{lemma:4.1}$.
Then we have
\begin{equation}\label{pri:4.2}
\begin{aligned}
\bigg|\int_{\Omega} q\nabla_\alpha \eta h^\alpha dx\bigg|
\leq C\big\|\eta\big\|_{C^{0,1}(\partial\Omega)}\Big(\int_{\Omega} |q|^2 \delta(x) dx\Big)^{\frac{1}{2}}
\Bigg\{\Big(\int_{\Omega} |\nabla h|^2 \delta(x) dx\Big)^{\frac{1}{2}}
+
\Big(\int_{\partial\Omega} |(h)^*|^2 dS\Big)^{\frac{1}{2}}
\Bigg\}.
\end{aligned}
\end{equation}
Moreover, if we additionally assume $|u|\in L^\infty(\Omega)$, then there holds
\begin{equation}\label{pri:4.5}
\begin{aligned}
\bigg|\int_{\Omega} q\nabla_\alpha \eta h^\alpha dx\bigg|
\leq C\big\|u\big\|_{L^\infty(\Omega)}\big\|\eta\big\|_{H^1(\partial\Omega)}
\big\|h\big\|_{L^2(\partial\Omega)},
\end{aligned}
\end{equation}
where  $C$ depends only on $d$ and $\Omega$.
\end{lemma}

\begin{proof}
We use the same notation as in the proof of Lemma $\ref{lemma:4.1}$, and
by the same localization methods as stated there, it suffices to
establish the estimate $\eqref{pri:4.2}$ with $\Omega$ replaced by $D$, under assumption that
$\Delta u = \nabla q$ and $\text{div}(u) = 0$ in $D$. In view of Lemma $\ref{lemma:2.1}$, it is known
that there exist an extension of $\eta$, still denoted by $\eta$, and $|\nabla^2\eta(x^\prime,t)|^2tdx^\prime dt$
is a Carleson measure. By
\begin{equation*}
\int_D q \nabla_\alpha \eta h^\alpha dx
=
\int_{\mathbb{R}^{d}_+} q\circ\rho \frac{\partial \eta}{\partial x_\alpha}\circ\rho h^\alpha\circ\rho \frac{\partial\varphi}{\partial t}dx^\prime dt := A,
\end{equation*}
where $\rho:\mathbb{R}_+^d\to D$ is referred as a special bi-Lipschitz map (see $\eqref{eq:4.1}$),
we manage to show
\begin{equation}\label{f:4.16}
\begin{aligned}
\bigg|\int_{\mathbb{R}^{d}_+} q\circ\rho \frac{\partial \eta}{\partial x_\alpha}
&\circ\rho h^\alpha\circ\rho \frac{\partial\varphi}{\partial t}dx^\prime dt\bigg|
\leq C\big\|\eta\big\|_{C^{0,1}(\partial\Omega)}\Big(\int_{\mathbb{R}^d_+} t|q\circ\rho|^2 dx^\prime dt\Big)^{\frac{1}{2}}\\
&\qquad\qquad \times
\Bigg\{\Big(\int_{\mathbb{R}^d_+} t|\nabla h\circ\rho|^2 dx^\prime dt\Big)^{\frac{1}{2}}
+
\Big(\int_{\mathbb{R}^{d-1}} |(h\circ\rho)^*|^2 dx^\prime\Big)^{\frac{1}{2}}
\Bigg\}
\end{aligned}
\end{equation}
and the desired estimate $\eqref{pri:4.2}$ will follow immediately. Observing that
\begin{equation}\label{f:4.13}
\begin{aligned}
A & = -\int_{\mathbb{R}_+^d} t\frac{\partial}{\partial t}\Bigg\{
q\circ\rho \frac{\partial \eta}{\partial x_\alpha}\circ\rho h^\alpha\circ\rho \frac{\partial\varphi}{\partial t}\Bigg\} dx^\prime dt \\
& = -\int_{\mathbb{R}_+^d} t\frac{\partial q}{\partial y}\circ\rho \frac{\partial \eta}{\partial x_\alpha}\circ\rho h^\alpha\circ\rho \Big(\frac{\partial\varphi}{\partial t}\Big)^2dx^\prime dt \\
& -\int_{\mathbb{R}_+^d} tq\circ\rho \frac{\partial^2 \eta}{\partial x_\alpha\partial y}\circ\rho h^\alpha\circ\rho \Big(\frac{\partial\varphi}{\partial t}\Big)^2dx^\prime dt
-\int_{\mathbb{R}_+^d} tq\circ\rho \frac{\partial \eta}{\partial x_\alpha}\circ\rho
\frac{\partial}{\partial t}\Big\{h^\alpha\circ\rho \frac{\partial\varphi}{\partial t}\Big\}dx^\prime dt,
\end{aligned}
\end{equation}
the last line of $\eqref{f:4.13}$ is controlled by
\begin{equation*}
C\big\|\eta\big\|_{C^{0,1}(\partial\Omega)}\Big(\int_{\mathbb{R}^d_+} t|q\circ\rho|^2 dx^\prime dt\Big)^{\frac{1}{2}}\\
\Bigg\{\Big(\int_{\mathbb{R}^d_+} t|\nabla h\circ\rho|^2 dx^\prime dt\Big)^{\frac{1}{2}}
+
\Big(\int_{\mathbb{R}^{d-1}} |(h\circ\rho)^*|^2 dx^\prime\Big)^{\frac{1}{2}}
\Bigg\},
\end{equation*}
where we use the fact that $|\nabla^2 \eta(x^\prime,t)|^2tdx^\prime dt$ and
$|\nabla^2 \varphi(x^\prime,t)|^2tdx^\prime dt$ are Carleson measures,
as well as \cite[Corollary 7.3.6]{JD}. The relatively tough term is in the second line of $\eqref{f:4.13}$, and integrating by parts in $t$ again, it is equal to
\begin{equation*}
\int_{\mathbb{R}_+^d} t^2\frac{\partial^2 q}{\partial y^2}\circ\rho \frac{\partial \eta}{\partial x_\alpha}\circ\rho h^\alpha\circ\rho \Big(\frac{\partial\varphi}{\partial t}\Big)^3dx^\prime dt
+ \int_{\mathbb{R}_+^d} t^2\frac{\partial q}{\partial y}\circ\rho
\frac{\partial}{\partial t}\Bigg\{\frac{\partial \eta}{\partial x_\alpha}\circ\rho h^\alpha\circ\rho \Big(\frac{\partial\varphi}{\partial t}\Big)^2\Bigg\}dx^\prime dt := B + E.
\end{equation*}

Let $\tilde{C} = C\|\eta\|_{C^{0,1}(\partial\Omega)}$.
We first handle $E$, which will produce the Carleson measures, and then it follows that
\begin{equation}\label{f:4.14}
\begin{aligned}
|E|
&\leq
\tilde{C}\Big(\int_{\mathbb{R}^d_+} t^3|\nabla q\circ\rho|^2 dx^\prime dt\Big)^{\frac{1}{2}}
\Bigg\{\Big(\int_{\mathbb{R}^d_+} t|\nabla h\circ\rho|^2 dx^\prime dt\Big)^{\frac{1}{2}}
+
\Big(\int_{\mathbb{R}^{d-1}} |(h\circ\rho)^*|^2 dx^\prime\Big)^{\frac{1}{2}}
\Bigg\}\\
&\leq
\tilde{C}\Big(\int_{\mathbb{R}^d_+} t|q\circ\rho|^2 dx^\prime dt\Big)^{\frac{1}{2}}
\Bigg\{\Big(\int_{\mathbb{R}^d_+} t|\nabla h\circ\rho|^2 dx^\prime dt\Big)^{\frac{1}{2}}
+
\Big(\int_{\mathbb{R}^{d-1}} |(h\circ\rho)^*|^2 dx^\prime\Big)^{\frac{1}{2}}
\Bigg\},
\end{aligned}
\end{equation}
where we apply the estimate $\eqref{pri:2.5}$ to the last step. To control $|B|$,
we apply the following identity
\begin{equation}\label{f:4.25}
\frac{\partial^2 q}{\partial y^2}\circ\rho
=\frac{1}{1+|\nabla_{x^\prime}\varphi|}\frac{\partial}{\partial x_k}
\bigg\{\frac{\partial q}{\partial y}\circ\rho \frac{\partial\varphi}{\partial x_k}
-\frac{\partial q}{\partial x_k}\circ\rho\bigg\}
- \frac{1}{1+|\nabla_{x^\prime}\varphi|}
\frac{\partial q}{\partial y}\circ\rho
\frac{\partial^2\varphi}{\partial x_k^2}
\end{equation}
to the term $B$ by noting the fact that $\Delta q = 0$ in $D$,  and then it is not hard to derive
\begin{equation*}
|B|\leq \tilde{C}\Big(\int_{\mathbb{R}^d_+} t|q\circ\rho|^2 dx^\prime dt\Big)^{\frac{1}{2}}
\Bigg\{\Big(\int_{\mathbb{R}^d_+} t|\nabla h\circ\rho|^2 dx^\prime dt\Big)^{\frac{1}{2}}
+
\Big(\int_{\mathbb{R}^{d-1}} |(h\circ\rho)^*|^2 dx^\prime\Big)^{\frac{1}{2}}
\Bigg\}.
\end{equation*}
This together with the estimate $\eqref{f:4.14}$ implies the desired result $\eqref{f:4.16}$.

We now turn to the proof of $\eqref{pri:4.5}$. Before proceeding further,
let $G\circ\rho$ be the harmonic extension of $\eta\circ\rho$ to $\mathbb{R}_+^d$,
i.e., $\Delta G\circ\rho = 0$ in $\mathbb{R}^d_+$ and $G= \eta$ on $\mathbb{R}^{d-1}$. We mention that
by a partition of unity we may assume that $\eta$ has compact support in $\mathbb{R}^{d-1}$.
Due to \cite{BD2,DJSK},
there holds
\begin{equation}\label{f:4.24}
\Big(\int_{\mathbb{R}^{d}_+}|\nabla^2 G\circ\rho|^2t dx^\prime dt\Big)^{\frac{1}{2}}
\leq C\Big(\int_{\mathbb{R}^{d-1}}|(\nabla G\circ\rho)^*|^2 dx^\prime\Big)^{\frac{1}{2}}
\end{equation}
where $C$ depends only on $d$. As in Remark $\ref{re:2.2}$, the harmonic extension function of $\eta$ is still denoted by itself in the follow statements.

Recalling the identity $\eqref{f:4.13}$, the last line of
$\eqref{f:4.13}$ is bounded by
\begin{equation*}
\begin{aligned}
&C\Bigg\{
\Big(\int_{\mathbb{R}^d_+}t|h\circ\rho|^2|q\circ\rho|^2dx^\prime dt\Big)^{\frac{1}{2}}
\Big(\int_{\mathbb{R}^d_+}t|\nabla^2\eta\circ\rho|^2dx^\prime dt\Big)^{\frac{1}{2}} \\
&\quad + \Big(\int_{\mathbb{R}^d_+}t|\nabla \eta\circ\rho|^2|q\circ\rho|^2dx^\prime dt\Big)^{\frac{1}{2}}
\Big(\int_{\mathbb{R}^d_+}t|\nabla h\circ\rho|^2dx^\prime dt\Big)^{\frac{1}{2}} \\
&\quad + \Big(\int_{\mathbb{R}^d_+}t|\nabla \eta\circ\rho|^2|q\circ\rho|^2dx^\prime dt\Big)^{\frac{1}{2}}
\Big(\int_{\mathbb{R}^d_+}t|h\circ\rho|^2|\nabla^2 \varphi(x^\prime,t)|^2 dx^\prime
dt\Big)^{\frac{1}{2}}\Bigg\}.
\end{aligned}
\end{equation*}
On account of Lemma $\ref{lemma:2.2}$ and the estimate $\eqref{f:4.24}$, the above quantities is controlled by
\begin{equation}\label{f:4.26}
 C\big\|u\circ\rho\big\|_{L^\infty(\mathbb{R}^d_+)}\Big(\int_{\mathbb{R}^{d-1}}|(h\circ\rho)^*|^2 dx^\prime\Big)^{\frac{1}{2}}
 \Big(\int_{\mathbb{R}^{d-1}}|(\nabla \eta\circ\rho)^*|^2 dx^\prime\Big)^{\frac{1}{2}},
\end{equation}
where we also use the fact that $|\nabla^2 \varphi(x^\prime,t)|^2tdx^\prime dt$ is a Carleson measure.
As we did in the proof of $\eqref{pri:4.2}$,
we have divided the integral in the second line of $\eqref{f:4.13}$ into two parts, also denoted by
$B$ and $E$, respectively. Then by noting that $|q\circ\rho|^2t dx^\prime dt$ and $|\nabla^2 \varphi(x^\prime,t)|^2tdx^\prime dt$ are Carleson measures, it is not hard to obtain that
\begin{equation*}
|E|\leq  C\big\|u\circ\rho\big\|_{L^\infty(\mathbb{R}^d_+)}\Big(\int_{\mathbb{R}^{d-1}}|(h\circ\rho)^*|^2 dx^\prime\Big)^{\frac{1}{2}}
 \Big(\int_{\mathbb{R}^{d-1}}|(\nabla \eta\circ\rho)^*|^2 dx^\prime\Big)^{\frac{1}{2}},
\end{equation*}
where the estimate $\eqref{pri:2.5}$ is also used in the computation. Furthermore,
the identity $\eqref{f:4.25}$ is also applied to estimate $|B|$ with the major change being
the substitution of using the Carleson measure $|\nabla^2 \eta\circ\rho|^2tdx^\prime dt$ for
employing the Carleson measure $|q\circ\rho|^2tdx^\prime dt$.
Thus, the quantity $|B|$ is also bounded by $\eqref{f:4.26}$ with a different constant $C$.
Up to now, we have proved that
\begin{equation*}
\bigg|\int_{D} q\nabla_\alpha \eta h^\alpha dx\bigg|
\leq C\big\|u\big\|_{L^\infty(\Omega)}\big\|\eta\big\|_{H^1(\Delta)}
\big\|h\big\|_{L^2(\Delta)},
\end{equation*}
where $\Delta = \{(x^\prime,y)\in\mathbb{R}^d:y=\psi(x^\prime), |x^\prime|<r\}
=\partial\Omega\cap B(0,r)$,
and we used the nontangential maximal function estimates (see \cite{V,EBFCEKGCV}).
The details is left to the reader and we have completed the proof.
\end{proof}

\begin{lemma}
Let $(h,\pi)\in
H^1(\Omega;\mathbb{R}^d)\times L^2(\Omega)$ be the solution of
$\Delta h = \nabla \pi$ and $\emph{div}(h) = 0$ in $\Omega$, and
$(h)^*\in L^2(\partial\Omega)$. Assume $\eta\in C^{0,1}(\partial\Omega)$, and
the vector-valued function $u$ is given as in Lemma $\ref{lemma:4.2}$.
Then we have
\begin{equation}\label{pri:4.3}
\begin{aligned}
\bigg|\int_{\Omega} \pi\nabla_\alpha \eta u^\alpha dx\bigg|
\leq C\big\|\eta\big\|_{C^{0,1}(\partial\Omega)}\Big(\int_{\Omega} |\pi|^2\delta(x) dx\Big)^{\frac{1}{2}}
\Bigg\{\Big(\int_{\Omega} |\nabla u|^2 \delta(x) dx\Big)^{\frac{1}{2}}
+
\Big(\int_{\partial\Omega} |(u)^*|^2 dS\Big)^{\frac{1}{2}}
\Bigg\}.
\end{aligned}
\end{equation}
Moreover, if $|u|\in L^\infty(\Omega)$, then there admits
\begin{equation}\label{pri:4.4}
\begin{aligned}
\bigg|\int_{\Omega} \pi\nabla_\alpha \eta u^\alpha dx\bigg|
\leq C\big\|u\big\|_{L^\infty(\Omega)}\big\|\eta\big\|_{H^1(\partial\Omega)}
\big\|h\big\|_{L^2(\partial\Omega)},
\end{aligned}
\end{equation}
where  $C$ depends only on $d$ and $\Omega$.
\end{lemma}

\begin{proof}
Using the same argument as in the proof of Lemma $\ref{lemma:4.2}$, we can
easily obtain the estimate $\eqref{pri:4.3}$ by interchanging $(u,q)$ and $(h,\pi)$, respectively.
The proof of $\eqref{pri:4.4}$ is quite similar to that given previously for the estimate
$\eqref{pri:4.5}$, and so is omitted.
\end{proof}

\begin{flushleft}
\textbf{Proof of Theorem \ref{thm:1.1}.} It follows from the identity $\eqref{pri:2.1}$ that
\end{flushleft}
\begin{equation}\label{f:4.17}
\begin{aligned}
\bigg|\int_{\partial\Omega}\big[\Lambda,\eta\big]&f\cdot h dS\bigg|
= \bigg|\int_{\partial\Omega} [\Lambda(\eta f)]^\alpha h^\alpha dS - \int_{\partial\Omega} \eta[\Lambda(f)]^\alpha h^\alpha dS \bigg|\\
& \leq \bigg|\int_{\Omega} u^\alpha \nabla \eta \cdot \nabla h^\alpha dx\bigg|
+  \bigg|\int_{\Omega} \nabla u^\alpha \cdot \nabla \eta h^\alpha dx\bigg|
+ \bigg|\int_{\Omega} q\nabla_\alpha \eta  h^\alpha\bigg|
+ \bigg|\int_{\Omega} \pi\nabla_\alpha \eta u^\alpha dx\bigg|\\
&:= I_1+I_2+I_3+I_4.
\end{aligned}
\end{equation}
Note that
\begin{equation}\label{f:4.18}
\begin{aligned}
I_3+I_4
&\leq C\|\eta\|_{C^{0,1}(\partial\Omega)}\|f\|_{L^2(\partial\Omega)}\|h\|_{L^2(\partial\Omega)}
\end{aligned}
\end{equation}
where we employ the estimates $\eqref{pri:4.2}$, $\eqref{pri:4.3}$, $\eqref{pri:2.4}$ and
$\eqref{pri:2.5}$, as well as $\|(h)^*\|_{L^2(\partial\Omega)}\leq C\|h\|_{L^2(\partial\Omega)}$
(see \cite[Theorem 3.9]{EBFCEKGCV}). Concerning $I_1$ and $I_2$, the estimates are based upon
the so-called Dahlberg's bilinear estimate, i.e. Lemma $\ref{lemma:4.1}$. Taking $I_1$ as an example,
let $v^\alpha= \nabla \eta u^\alpha$ in the estimate $\eqref{pri:4.1}$, and it follows that
\begin{equation}\label{f:4.19}
\begin{aligned}
&\qquad\qquad\qquad\quad\bigg|\int_{\Omega} \nabla h \cdot \nabla \eta u dx\bigg|
\leq C
\Bigg\{ \Big(\int_{\Omega} |\nabla h|^2 \delta(x) dx\Big)^{\frac{1}{2}}
+ \int_{\Omega} |\pi|^2 \delta(x) dx\Big)^{\frac{1}{2}}
\Bigg\} \\
&\times
\Bigg\{
\|\eta\|_{C^{0,1}(\partial\Omega)}\Big(\int_{\Omega} |\nabla u|^2\delta(x) dx\Big)^{\frac{1}{2}}
+\Big(\int_{\Omega} |u|^2|\nabla^2 \eta|^2 \delta(x) dx\Big)^{\frac{1}{2}}
+\|\eta\|_{C^{0,1}(\partial\Omega)}\Big(\int_{\partial\Omega} |(u)^*|^2 dS\Big)^{\frac{1}{2}}
\Bigg\} \\
&\qquad\qquad\qquad\qquad\qquad\qquad\qquad~
\leq C\big\|\eta\big\|_{C^{0,1}(\partial\Omega)}
\big\|f\big\|_{L^2(\partial\Omega)}\big\|h\big\|_{L^2(\partial\Omega)},
\end{aligned}
\end{equation}
where we mention that $|\nabla^2 \eta|^2 \delta(x)dx$ is a Carleson measure. The term $I_2$ follows
the similar computations. Collecting the estimates $\eqref{f:4.17}$, $\eqref{f:4.18}$
and $\eqref{f:4.19}$ consequently leads to
\begin{equation*}
\bigg|\int_{\partial\Omega}\big[\Lambda,\eta\big]f\cdot h dS\bigg|
\leq  C\big\|\eta\big\|_{C^{0,1}(\partial\Omega)}
\big\|f\big\|_{L^2(\partial\Omega)}\big\|h\big\|_{L^2(\partial\Omega)},
\end{equation*}
which yields the desired estimate $\eqref{pri:1.1}$ by duality.

Now, we proceed to prove the estimate $\eqref{pri:1.2}$ in the case of $d=3$. In such case, it is
well-known that
\begin{equation*}\label{f:4.23}
\|u\|_{L^\infty(\Omega)}
\leq C\|f\|_{L^\infty(\partial\Omega)}
\end{equation*}
(see \cite[Theorem 0.2]{Shen3}), and this indicates it suffices to establish
\begin{equation}\label{f:4.20}
\big\|\Lambda(\eta f) - \eta\Lambda(f)\big\|_{L^2(\partial\Omega)}
\leq C\|\eta\|_{H^{1}(\partial\Omega)}\|u\|_{L^\infty(\Omega)}.
\end{equation}

Before approaching the above estimate, let $G$ be the harmonic extension of $\eta$ to $\Omega$,
i.e., $\Delta G = 0$ in $\Omega$ and $G= \eta$ on $\partial\Omega$. Furthermore, due to \cite{BD2,DJSK},
there holds
\begin{equation}\label{f:4.15}
\Big(\int_{\partial\Omega}|(\nabla G)^*|^2 dS\Big)^{\frac{1}{2}}
+ \Big(\int_{\Omega}|\nabla^2 G|^2\delta(x) dx\Big)^{\frac{1}{2}}
\leq C\|\eta\|_{H^1(\partial\Omega)},
\end{equation}
where $C$ depends only on $d$ and $\Omega$. As in Remark $\ref{re:2.2}$, the harmonic extension function of $\eta$ is still denoted by itself in the follow statements.

To estimate $\eqref{f:4.15}$, let us review $\eqref{f:4.17}$, and re-estimate it under the new condition.
It follows the Dahlberg's bilinear estimate $\eqref{pri:4.1}$ that
\begin{equation}\label{f:4.21}
\begin{aligned}
I_1
&\leq C
\Bigg\{ \Big(\int_{\Omega} |\nabla h|^2 \delta(x) dx\Big)^{\frac{1}{2}}
+ \Big(\int_{\Omega} |\pi|^2 \delta(x) dx\Big)^{\frac{1}{2}}
\Bigg\} \\
&\qquad\times
\Bigg\{\Big(\int_{\Omega} |\nabla \eta|^2 |\nabla u|^2 \delta(x) dx\Big)^{\frac{1}{2}}
+\Big(\int_{\Omega} |u|^2|\nabla^2 \eta|^2 \delta(x) dx\Big)^{\frac{1}{2}}
+\Big(\int_{\partial\Omega} |(u\nabla \eta)^*|^2 dS\Big)^{\frac{1}{2}}
\Bigg\} \\
&\leq C\big\|h\big\|_{L^2(\partial\Omega)}
\Bigg\{\big\|u\big\|_{L^\infty(\Omega)}\big\|(\nabla \eta)^*\big\|_{L^2(\partial\Omega)}
+\big\|u\big\|_{L^\infty(\Omega)}\Big(\int_{\Omega}|\nabla^2 \eta|^2\delta(x) dx\Big)^\frac{1}{2}
\Bigg\}\\
&\leq
C\big\|h\big\|_{L^2(\partial\Omega)}\|\eta\big\|_{H^1(\partial\Omega)}\big\|u\big\|_{L^\infty(\Omega)}.
\end{aligned}
\end{equation}
In the second inequality, we employ the fact that $|\nabla u|^2\delta(x)dx$ is a Carleson measure
(see Lemma $\ref{lemma:2.2}$), and
the last inequality follows from the estimate $\eqref{f:4.15}$. For $I_2$, we first observe that
\begin{equation*}
\int_\Omega \nabla u^\alpha\cdot \nabla \eta h^\alpha dx
= - \int_\Omega u^\alpha \nabla \eta\cdot \nabla h^\alpha dx +\int_{\partial\Omega}
n\cdot\nabla \eta u^\alpha h^\alpha dS
\end{equation*}
since $\Delta \eta = 0$ in $\Omega$ here. Hence, it is not hard to obtain
\begin{equation}\label{f:4.22}
I_2 \leq
C\big\|h\big\|_{L^2(\partial\Omega)}\|\eta\big\|_{H^1(\partial\Omega)}\big\|u\big\|_{L^\infty(\Omega)},
\end{equation}
where we use the so-called Rellich estimate $\|\nabla_{\tan}\eta\|_{L^2(\partial\Omega)}\approx\|\frac{\partial \eta}{\partial n}\|_{L^2(\partial\Omega)}$
(see \cite{V}). Finally, it follows from the estimates $\eqref{pri:4.5}$ and $\eqref{pri:4.4}$ that
\begin{equation*}
I_3 + I_4 \leq
C\big\|h\big\|_{L^2(\partial\Omega)}\|\eta\big\|_{H^1(\partial\Omega)}\big\|u\big\|_{L^\infty(\Omega)},
\end{equation*}
and this together with $\eqref{f:4.21}$ and $\eqref{f:4.22}$ leads to the desired estimate
$\eqref{f:4.20}$ by a duality argument. We have completed the proof.
\qed

\section{Appendix}\label{section:5}

In this section, we give an simple and illuminating verification of $\eqref{pri:1.1}$. Consider the following Stokes system:
\begin{equation}\label{pde:2.2}
\left\{\begin{aligned}
(\partial^2_x + \partial^2_y)u^1 &=\partial_x q \\
(\partial^2_x + \partial^2_y)u^2 &=\partial_y q
\end{aligned}\right.
\quad \text{and}
\quad
\partial_x u^1 + \partial_y u^2 = 0
\quad \text{in}~\mathbb{R}^2_+,
\qquad
\left\{\begin{aligned}
u^1 &= f^1 \\
u^2 &= f^2
\end{aligned}\right.
\quad \text{on}~\partial\mathbb{R}^2_+ = \mathbb{R}.
\end{equation}
Additionally, it is convenient to assume $u$ will vanish as $|x|$ goes to infinity.
To solve the above equations, let $u=\nabla^T\psi$, where
$\nabla^T = (-\partial_y,\partial_x)$, and $\psi$ is a scale function.
Plugging it back into $\eqref{pde:2.2}$, it is not hard to derive  $\Delta^2\psi = 0$ in $\mathbb{R}^2_+$. Then applying Fourier transformation with respect to $x$, we have
\begin{equation}\label{pde:2.3}
\left\{\begin{aligned}
(\partial_y^2 - |k|^2)^2\widehat{\psi}(k,y) &= 0
& \quad &\text{in} ~~\mathbb{R}\times\mathbb{R}_+, \\
\widehat{\psi}(k,0) &= \widehat{f^2}/(ik)& \quad &\text{on} ~~\mathbb{R}, \\
\partial_y\widehat{\psi}(k,0) &= -\widehat{f^1}& \quad &\text{on} ~~\mathbb{R}.
\end{aligned}\right.
\end{equation}
By the condition $\widehat{\psi}(k,\infty) = 0$, it is clear to figure out the solution of  $\eqref{pde:2.3}$, and it
is written by
\begin{equation}\label{f:2.11}
\widehat{\psi}(k,y) = \frac{\widehat{f^2}}{ik}e^{-|k|y}
+(|k|\frac{\widehat{f^2}}{ik}-\widehat{f^1})y e^{-|k|y}.
\end{equation}
Since
\begin{equation*}
 \big(\Lambda(f)\big)^1 = \frac{\partial u^1}{\partial y} - q
 \qquad
 \text{and}
 \qquad
 \big(\Lambda(f)\big)^2 = \frac{\partial u^2}{\partial y},
\end{equation*}
we insert $(u^1,u^2) = (-\partial_y\psi,\partial_x\psi)$ into the above formula, and then taking
Fourier transformation, we obtain
\begin{equation*}
\widehat{\Lambda(f)} = \big(\partial_y \widehat{u^1}-\widehat{q},~\partial_y\widehat{u^2}\big)
= \big(-\partial^2_y\widehat{\psi}-\widehat{q},~ik\partial_y\widehat{\psi}\big)\Big|_{y=0}.
\end{equation*}
Hence the problem is reduced to calculate the quantities $\partial_y^2\widehat{\psi}, \partial_y\widehat{\psi}$ and $\widehat{q}$ on $\mathbb{R}$. By a tedious computation, it follows
from $\eqref{f:2.11}$ that
\begin{equation}\label{f:2.12}
\begin{aligned}
\partial_y^2\widehat{\psi}\Big|_{y=0} &= 2|k|\widehat{f^1} + ik\widehat{f^2}, \\
ik\partial_y\widehat{\psi}\Big|_{y=0} &= -ik\widehat{f^1}.
\end{aligned}
\end{equation}
The rest thing is to compute $\widehat{q}$. In view of $\eqref{pde:2.2}$ and $u^1=-\partial_y\psi$, we have $(\partial^2_y-|k|^2)\partial_y\widehat{\psi} = ik\widehat{q}$ in $\mathbb{R}_+^2$, and
by $\eqref{f:2.11}$, there holds
\begin{equation}\label{f:2.13}
\widehat{q}\Big|_{y=0} = -2|k|\widehat{f^2} + 2ik\widehat{f^1}.
\end{equation}
Hence, combining $\eqref{f:2.12}$ and $\eqref{f:2.13}$, it is clear to see that
the quantity $\widehat{\Lambda(f)}$ is determined by
\begin{equation*}
\big( -2|k|\widehat{f^1} - ik\widehat{f^2} + 2|k|\widehat{f^2} - 2ik\widehat{f^1},
\quad -ik\widehat{f^1}\big).
\end{equation*}

Let $H$ denote the Hilbert transform. It is well known that $\widehat{H(h)}(k) = -i \text{sgn}(k)\widehat{h}(k)$ for any $h$ in Schwartz class, and by observing $|k| = \text{sgn}(k)k$ we may have
\begin{equation}
\Lambda(f) = \big(-2H(\partial_xf^1)-\partial_x f^2 + 2H(\partial_xf^2)
-\partial_x f^1,\quad -\partial_x f^1 \big)
\end{equation}
Hence, for any $\eta\in C^{0,1}(\mathbb{R})$, we may directly compute the quantity
$\Lambda(\eta f)-\eta \Lambda(f)$. Since $\partial_x(\eta f^1) - \eta\partial_xf^1 = f^1\partial_x \eta $, we only
study its first component
\begin{equation*}
\begin{aligned}
\big(\Lambda(\eta f)-\eta\Lambda(f)\big)^1
& = -2H(\partial_x(\eta f^1)) - \partial_x (\eta f^2) + 2H(\partial_x(\eta f^2)) -  \partial_x (\eta f^1) \\
& \quad + 2\eta H(\partial_x f^1) + \eta\partial_x f^2 - 2\eta H(\partial_xf^2) + \eta\partial_x f^1 \\
& = -2\big[H(\partial_x(\eta f^1))-\eta H(\partial_x f^1)\big]
+2\big[H(\partial_x(\eta f^2))-\eta H(\partial_x f^2)\big] - \partial_x\eta(f^1+f^2).
\end{aligned}
\end{equation*}
Thus from the estimate
\begin{equation}\label{f:2.14}
\|H(\partial_x(\eta f^i))-\eta H(\partial_x f^i)\|_{L^2(\mathbb{R})}
\leq C\|\eta\|_{C^{0,1}(\mathbb{R})}\|f^i\|_{L^2(\mathbb{R})}
\end{equation}
where $i=1,2$, we arrive at
\begin{equation}\label{f:2.15}
\begin{aligned}
\|\Lambda(\eta f)-\eta \Lambda(f)\|_{L^2(\mathbb{R})}
&\leq C\Big\{\sum_{i=1}^2\|H(\partial_x(\eta f^{i}))-\eta H(\partial_x f^{i})\|_{L^2(\mathbb{R})}+\|\nabla \eta \|_{L^\infty(\mathbb{R})}\|f\|_{L^2(\mathbb{R})}\Big\} \\
&\leq C\|\eta \|_{C^{0,1}(\mathbb{R})}\|f\|_{L^2(\mathbb{R})}.
\end{aligned}
\end{equation}
Our task now is to estimate $\eqref{f:2.14}$. Although the proof is probably known to experts in the
area, we provide it here for the sake of the completeness.
\begin{equation*}
H(\partial_x(\eta f^i))-\eta H(\partial_x f^i) = H(\partial_x\eta f^i)
+H\big((\eta-\eta(x))\partial_xf^i\big)
\end{equation*}
Note that
\begin{equation*}
\begin{aligned}
H\big((\eta-\eta(x))\partial_xf^i\big)(x)
&=\frac{1}{\pi}\lim_{\varepsilon\to0}\int_{|z|>\varepsilon}\frac{\eta(z)-\eta(x)}{x-z}\partial_z f^i(z) dz \\
&= -\frac{1}{\pi}\lim_{\varepsilon\to0}\int_{|z|>\varepsilon}\frac{\partial_z\eta(z)f^i(z)}{x-z}dz
+ \frac{1}{\pi}\lim_{\varepsilon\to0}\int_{|z|>\varepsilon}\frac{\eta(z)-\eta(x)}{(x-z)^2}f^i(z)dz,
\end{aligned}
\end{equation*}
and this together with
\begin{equation*}
\eta(z)-\eta(x) = \int_0^1 \partial_\xi \eta(\xi)dt\cdot(z-x), \qquad \xi = tz+(1-t)x,
\end{equation*}
implies the desired estimate $\eqref{f:2.14}$ (see \cite{JD}). Indeed, the estimate $\eqref{f:2.15}$ may hold for
any $1<p<\infty$, i.e.,
\begin{equation*}
\|\Lambda(\eta f)-\eta\Lambda(f)\|_{L^p(\mathbb{R})}
\leq C\|\eta\|_{C^{0,1}(\mathbb{R})}\|f\|_{L^p(\mathbb{R})}.
\end{equation*}

In the case of $d\geq 3$, the previous proof indicates that the Hilbert transform will be replaced by Riesz transforms, and it will provide another proof for Theorem $\ref{thm:1.1}$ in the special case of
$\mathbb{R}_+^d$. Furthermore, if the domain $\Omega$ is sufficient smooth, this approach may be applied to the following estimate
\begin{equation*}
\|\Lambda(\eta f)-\eta\Lambda(f)\|_{L^p(\partial\Omega)}
\leq C\|\eta\|_{C^{0,1}(\partial\Omega)}\|f\|_{L^p(\partial\Omega)},
\end{equation*}
and we will complete this topic through pseudodifferential operator arguments in a separate work.

\begin{center}
\textbf{Acknowledgements}
\end{center}

The first author wants to express his sincere appreciation to Professor Zhongwei Shen and Professor Peihao Zhao for their encouragements and instructions. The first author was supported by the National Natural Science Foundation of China (Grant No.11471147). The third author was supported by the National Natural Science
Foundation of China (Grant No.11571020).

\end{document}